\documentclass[12pt]{amsart}

\usepackage[cp1251]{inputenc}
\usepackage{amsthm,amssymb,amsmath,amsfonts,wasysym}
\usepackage{textcomp} 
\usepackage{etoolbox}
\usepackage{comment}

\usepackage{graphicx}
\usepackage[matrix,arrow,curve]{xy}
\usepackage{longtable}
\newcommand{\Z}{\ensuremath{\mathbb{Z}}}

\newcommand{\A}{\ensuremath{\mathbb{A}}}

\newcommand{\F}{\ensuremath{\mathbb{F}}}

\newcommand{\CG}{\ensuremath{\mathfrak{C}}}

\newcommand{\DG}{\ensuremath{\mathfrak{D}}}

\newcommand{\AG}{\ensuremath{\mathfrak{A}}}

\newcommand{\SG}{\ensuremath{\mathfrak{S}}}

\newcommand{\ka}{\ensuremath{\Bbbk}}

\newcommand{\kka}{\ensuremath{\overline{\Bbbk}}}

\newcommand{\XX}{\ensuremath{\overline{X}}}

\newcommand{\ii}{\ensuremath{\mathrm{i}}}

\newcommand{\Pro}{\ensuremath{\mathbb{P}}}

\newcommand{\Aut}{\ensuremath{\operatorname{Aut}}}

\newcommand{\Gal}{\ensuremath{\operatorname{Gal}}}

\newcommand{\Pic}{\ensuremath{\operatorname{Pic}}}

\newcommand{\Char}{\ensuremath{\operatorname{char}}}

\newcommand{\ord}{\ensuremath{\operatorname{ord}}}

\makeatletter
\@addtoreset{equation}{section}
\makeatother

\newtheorem{theorem}[equation]{Theorem}
\newtheorem{proposition}[equation]{Proposition}
\newtheorem{lemma}[equation]{Lemma}
\newtheorem{corollary}[equation]{Corollary}

\theoremstyle{definition}
\newtheorem{example}[equation]{Example}
\newtheorem{definition}[equation]{Definition}
\newtheorem{cordef}[equation]{Corollary-definition}

\theoremstyle{remark}
\newtheorem{remark}[equation]{Remark}
\newtheorem{notation}[equation]{Notation}
\newtheorem{question}[equation]{Question}

\title{Quotients of del Pezzo surfaces}
\address{Institute for Information Transmission Problems, 19 Bolshoy Karetnyi side-str., Moscow 127994, Russia}
\address{Laboratory of Algebraic Geometry, National Research University Higher School of Economics, 6 Usacheva str., Moscow 119048, Russia}
\email{trepalin@mccme.ru}
\thanks{This work is supported by the Russian Science Foundation under grant \textnumero 18-11-00121}
\author{Andrey Trepalin}

\begin{document}

\begin{abstract}
Let $\ka$ be any field of characteristic zero, $X$ be a del Pezzo surface and $G$ be a finite subgroup in $\Aut(X)$. In this paper we study when the quotient surface $X / G$ can be non-rational over $\ka$. Obviously, if there are no smooth $\ka$-points on $X / G$ then it is not $\ka$-rational. Therefore under assumption that the set of smooth $\ka$-points on $X / G$ is not empty we show that there are few possibilities for non-$\ka$-rational quotients.

The quotients of del Pezzo surfaces of degree $2$ and greater are considered in the author's previous papers. In this paper we study the quotients of del Pezzo surfaces of degree $1$. We show that they can be non-$\ka$-rational only for the trivial group or cyclic groups of order $2$, $3$ and $6$. For the trivial group and the group of order $2$ we show that both $X$ and $X / G$ are not $\ka$-rational if the $G$-invariant Picard number of $X$ is~$1$. For the groups of order $3$ and $6$ we construct examples of both $\ka$-rational and non-$\ka$-rational quotients \mbox{of both $\ka$-rational} and non-$\ka$-rational del Pezzo surfaces of degree $1$ such that the $G$-invariant Picard number of $X$ is~$1$.

As a result of complete classification of non-$\ka$-rational quotients of del Pezzo surfaces we classify surfaces that are birationally equivalent to quotients of $\ka$-rational surfaces, and obtain some corollaries concerning fields of invariants of $\ka(x , y)$.

\end{abstract}


\maketitle
\tableofcontents

\section{Introduction}

Let $\ka$ be an arbitrary field. An $n$-dimensional variety $V$ is called \textit{$\ka$-rational} if there exists a birational map $\Pro^n_{\ka} \dashrightarrow V$ defined over $\ka$. If there exists a dominant rational map $\Pro^n_{\ka} \dashrightarrow V$ defined over $\ka$, then $V$ is called \textit{$\ka$-unirational}. Obviously, any $\ka$-rational variety is $\ka$-unirational. A natural question arises: when a $\ka$-unirational variety is $\ka$-rational?

For $n = 1$ the answer is a classical result.

\begin{theorem}[J. L\"uroth \cite{Lur76}]
\label{Luroth}
Any $\ka$-unirational curve is $\ka$-rational.
\end{theorem}

For $n = 2$ the same result follows from Castelnuovo's rationality criterion (see~\cite{Cast94}), but there are two additional conditions on $\ka$: the field $\ka$ is algebraically closed, and~$\Char \ka = 0$.

\begin{theorem}[G. Castelnuovo]
\label{Castelnuovo}
Let $\ka$ be an algebraically closed field of characteristic zero. Then any $\ka$-unirational surface is $\ka$-rational.
\end{theorem} 

For $n \geqslant 3$ there are several constructions of $\ka$-unirational and non-$\ka$-rational varieties (see \cite{AM72}, \cite{CG72} and \cite{IM71}) even if $\ka$ is algebraically closed.

Note that if $V$ is $\ka$-rational then its rational function field is isomorphic \mbox{to $K = \ka(x_1, x_2, \ldots, x_n)$}, and if $V$ is $\ka$-unirational then its rational function field is a~subfield of $K$ of transcendence degree $n$. Therefore results about $\ka$-unirationality can be easily reformulated in terms of subfields of $K$. One important class of subfields of $K$ is fields of invariants $K^G$ under the action of a finite group $G$ of $\Aut_{\ka}(K)$. In this case $G$ acts on $\Pro^n_{\ka}$ by birational selfmaps. One can regularize this action, and find a $\ka$-rational variety $V$, such that $G$ faithfully acts on $V$ by automorphisms. The field of invariants~$K^G$ corresponds to the quotient $V / G$. One has~$K^G \cong \ka(y_1, y_2, \ldots, y_n)$ if and only if $V / G$ is $\ka$-rational. A variety birationally equivalent to a quotient of a $\ka$-rational variety by a finite group of its automorphisms is called \textit{Galois $\ka$-unirational}.

If $n = 1$ then any Galois $\ka$-unirational curve is $\ka$-rational by Theorem \ref{Luroth}, and \mbox{if~$n = 2$}, $\Char \ka = 0$ and $\ka$ is algebraically closed then any Galois $\ka$-unirational surface is $\ka$-rational by Theorem \ref{Castelnuovo}. For $n \geqslant 3$ an example of Galois $\ka$-unirational and non-$\ka$-rational variety is constructed in \cite[Appendix B]{CG72}. The aim of this paper is to study Galois $\ka$-unirational surfaces over algebraically nonclosed fields of characteristic $0$. Thus we study quotients of $\ka$-rational surfaces by finite automorphisms groups.

Examples of $\ka$-unirational and non-$\ka$-rational surfaces for algebraically nonclosed fields are well-known. For instance, any del Pezzo surface of degree $4$ is $\ka$-unirational (see \cite[Theorem~IV.7.8]{Man74}), but not all these surfaces are $\ka$-rational. Moreover, any such surface is birationally equivalent to a quotient of $\ka$-rational surface by an involution. Thus we want to find possibilities for finite groups $G$ acting on $\ka$-rational surfaces $S$, such that the quotient $S / G$ is not $\ka$-rational.

To start solving this problem we should introduce some notions of minimal model program. A smooth surface $S$ is called \textit{minimal} (resp. \textit{$G$-minimal}) if any birational morphism (resp. $G$-equivariant morphism) of smooth surfaces $S \rightarrow S'$ is an isomorphism. The following theorem is an important criterion of $\ka$-rationality over an arbitrary perfect field $\ka$.

\begin{theorem}[{\cite[Chapter 4]{Isk96}}]
\label{ratcrit}
A minimal rational surface $X$ over a perfect field $\ka$ is $\ka$-rational if and only if the following two conditions are satisfied:

(i) $X(\ka) \neq \varnothing$;

(ii) $K_X^2 \geqslant 5$.
\end{theorem}

If $S$ is a smooth surface and $G \subset \Aut_{\ka}(S)$ then there exists a $G$-equivariant morphism $S \rightarrow X$ such that $X$ is $G$-minimal. Moreover, the quotients $S / G$ and $X / G$ are birationally equivalent.
By \cite[Theorem 1]{Isk79} any $G$-minimal \textit{geometrically rational} (i.e. $\XX = X \otimes_{\ka} \kka$ is $\kka$-rational) surface $X$ either is a del Pezzo surface such that $\rho(X)^G = 1$, or admits a~structure of a $G$-equivariant conic bundle $X \rightarrow B$ such that~$\rho(X)^G = 2$ and $B$ is a~smooth curve with genus $0$.

Quotients of conic bundles are considered in \cite{Tr16a}. One of the main results of \cite{Tr16a} is the following.

\begin{theorem}[{\cite[cf. Theorem 1.3]{Tr16a}}]
\label{Cbundlebirat}

Let $\ka$ be a field of characteristic zero, $X$ be a $\ka$-rational surface admitting a structure of a conic bundle, and $G$ be a finite group acting on $X$. Then the quotient $X / G$ is $\ka$-birationally equivalent to a quotient of $\ka$-rational conic bundle by a group $H$, where $H$ is a cyclic group $\CG_{2^k}$ of order~$2^k$, dihedral group $\DG_{2^k}$ of order~$2^k$, alternating group~$\AG_4$ of degree $4$, symmetric group $\SG_4$ of degree $4$ or alternating group $\AG_5$ of degree $5$, and there is a surjective homomorphism $G \rightarrow H$.
\end{theorem}

\begin{remark}
\label{Cbundlebiratrem}
Actually, from the proof of \cite[cf. Theorem 1.3]{Tr16a} one can see that the condition of $\ka$-rationality of $X$ in Theorem \ref{Cbundlebirat} can be replaced by geometrical rationality.
\end{remark}

We need the following definition (see \cite{Al94}) to state the other main result of~\cite{Tr16a}.

\begin{definition}
\label{birunboundness}
A class $\mathfrak{V}$ of varieties is \textit{$\ka$-birationally bounded} if there is a morphism $\varphi: \mathcal{X} \rightarrow \mathcal{S}$ between algebraic schemes of finite type such that every member of $\mathfrak{V}$ \mbox{is $\ka$-birationally} equivalent to one of the geometric fibres of $\varphi$. We say that $\mathfrak{V}$ \mbox{is \textit{$\ka$-birationally unbounded}} if it is not $\ka$-birationally bounded.
\end{definition}

\begin{theorem}[{\cite[Theorem 1.8]{Tr16a}}]
\label{Unboundness}
Let $\ka$ be a field of characteristic zero such that not all elements of $\ka$ are squares and $G$ be a finite group of automorphisms of $\Pro^1_{\ka}$. Then the class of $G$-quotients of $\ka$-rational conic bundles is $\ka$-birationally unbounded in the following cases:

\begin{itemize}

\item if $\ka$ contains $\xi_k = e^{\frac{2\pi \ii}{k}}$ and $G$ is a cyclic group $\CG_{2k}$ of order $2k$;

\item if $\ka$ contains $\cos \frac{2\pi}{k}$ and $G$ is a dihedral group $\DG_{2k}$ of order $2k$;

\item if $\ka$ contains $\ii$ and $G$ is an alternating group $\AG_4$ of degree $4$;

\item if $\ka$ contains $\ii$ or $\ii \sqrt{2}$ and $G$ is a symmetric group $\SG_4$ of degree~$4$;

\item if $\ka$ contains $\ii$ and $\sqrt{5}$ and $G$ is an alternating group $\AG_5$ of degree $5$.

\end{itemize}

\end{theorem}

In particular, it means that if there is at least one non-square in $\ka$ then the class of Galois $\ka$-unirational surfaces is $\ka$-birationally unbounded. For example, if $\ka$ is the field of real numbers $\mathbb{R}$ then each conic bundle $X \rightarrow \Pro^1_{\mathbb{R}}$ such that $X(\mathbb{R}) \ne \varnothing$ is Galois \mbox{$\mathbb{R}$-unirational} by \cite[Corollary 4.4]{Isk67}. But the class of conic bundles is $\mathbb{R}$-birationally unbounded by \cite[Theorem 1.6]{Isk67}. On the other hand all quotients of conic bundles by finite automorphism groups are $\ka$-birationally equivalent to surfaces admitting a structure of conic bundle by \cite[Theorem 4.1]{Tr16a}.

The class of del Pezzo surfaces is $\ka$-birationally bounded itself. Therefore quotients of del Pezzo surfaces are $\ka$-birationally bounded too. But we want to find the possibilities, when such quotients are not $\ka$-rational, and classify these quotients up to $\ka$-birational isomorphisms.

Quotients of del Pezzo surfaces of degree greater than $1$ are considered in the author's papers \cite{Tr14}, \cite{Tr18a}, \cite{Tr16b} and \cite{Tr18b}. In this paper we study quotients of del Pezzo surfaces of degree $1$. We collect results about quotients of all del Pezzo surfaces in the following theorem.

\begin{theorem}
\label{main}
Let $\ka$ be a field of characteristic zero, $X$ be a del Pezzo surface of degree~$d = K_X^2$ over $\ka$, and $G$ be a finite subgroup of $\Aut_{\ka}(X)$ such that there is a smooth $\ka$-point on $X / G$. Then the quotient~$X / G$ is $\ka$-rational in all cases except the cases listed below. In those cases the quotient $X / G$ can be non-$\ka$-rational for a suitable field $\ka$:
\begin{enumerate}
\item $d \leqslant 4$, and $G$ is a trivial group;
\item $d = 4$, the order of $G$ is $2$ or $4$, and all nontrivial elements of $G$ have only isolated fixed points;
\item $d = 3$, the order of $G$ is $3$, and all nontrivial elements of $G$ have only isolated fixed points;
\item $d = 2$, the group $G$ has order $2$ or $3$, or $G$ is a non-abelian group of order $6$, and all nontrivial elements of $G$ have only isolated fixed points;
\item $d = 2$, the group $G$ has order $2$ or $4$, or $G$ is a non-abelian group of order $8$, and all nontrivial elements of $G$, except a unique involution, have only isolated fixed points;
\item $d = 1$, the group $G$ is a cyclic group of order $2$ or $6$, and all nontrivial elements of $G$, except a unique involution, have only isolated fixed points;
\item $d = 1$, and $G$ is a cyclic group of order $3$.
\end{enumerate}
Moreover, for these cases if a nontrivial element has a curve of fixed points then this curve is elliptic.

For each of those groups there exists an example of a non-$\ka$-rational quotient $X / G$ for a suitable field $\ka$.
\end{theorem}

\begin{remark}
\label{conditions}
We need a smooth $\ka$-point on $X / G$, to apply Theorem \ref{ratcrit}. Without this condition for the groups not listed in Theorem \ref{main} we can only say that the quotient~$X / G$ is birationally equivalent to a surface $Y$ such that $K_Y^2 \geqslant 5$. For del Pezzo surfaces of degree $d \geqslant 3$ this condition follows from the condition $X(\ka) \neq \varnothing$, since in this case $X$ is $\ka$-unirational (see \cite[Theorems~IV.7.8 and IV.8.1]{Man74}). Therefore $X(\ka)$ is dense, and the set of smooth $\ka$-points on $X / G$ is dense.

Also, for del Pezzo surfaces of degree $d \leqslant 5$ one can omit the condition that $G$ is finite, since in these cases $\Aut_{\ka}(X)$ is finite.
\end{remark}

To prove Theorem \ref{main} we consider quotients of del Pezzo surfaces starting from degree~$d = 9$ and finishing at $d = 1$. The degrees $d \geqslant 2$ are considered in \cite[Theorem 1.3]{Tr14}, \cite[Proposition 3.1]{Tr18a}, \cite[Corollary 1.4]{Tr14}, \cite[Proposition 4.1]{Tr18a}, \cite[Proposition 5.1]{Tr18a}, \cite[Theorem 1.3]{Tr16b} and \cite[Proposition 4.2]{Tr18b}. Also one can find there more precise descriptions of the groups mentioned in Theorem \ref{main}.

\begin{remark}
\label{geq5touse}
Actually, in \cite{Tr16a}, \cite{Tr14} and \cite{Tr18a} it is proved that if $X$ is a smooth geometrically rational surface over $\ka$ such that $K_X^2 \geqslant 5$, $X(\ka) \ne \varnothing$ and $G$ is a finite subgroup of $\Aut_{\ka}(X)$ then the quotient variety $X / G$ is $\ka$-rational (see \cite[Corollary~1.2]{Tr18a}).

Moreover, if $X$ is a smooth geometrically rational surface over $\ka$ such that $K_X^2 \geqslant 5$ and~$G$ is a finite subgroup of $\Aut_{\ka}(X)$ then the quotient~$X / G$ is birationally equivalent to a surface $Y$ such that $K_Y^2 \geqslant 5$.
\end{remark}

Quotients of del Pezzo surfaces of degree $1$ are considered in Proposition \ref{DP1}. Also one can find there more precise descriptions of the groups mentioned in Theorem \ref{main} for the cases with $d = 1$.

Note that in Theorem \ref{main} we do not assume that $X$ is $\ka$-rational. But we are interested in quotients of $\ka$-rational surface to classify Galois $\ka$-unirational surfaces, and obtain corollaries about fields of invariants of $\ka(x_1, x_2)$. Thus we want to study, which possibilities of Theorem \ref{main} hold for $\ka$-rational del Pezzo surfaces. For simplicity, we assume that $X$ is $G$-minimal since any other case can be simply reduced to this case.

\begin{theorem}
\label{mainex}
Let $\ka$ be a field of characteristic zero, $X$ be a $G$-minimal del Pezzo surface of degree $d = K_X^2$ over $\ka$, and $G$ be a finite subgroup mentioned in Theorem \ref{main}.

If $G$ is trivial, or $d = 2$, $|G| = 2$ and $G$ has only isolated fixed points, or $d = 1$, $|G| = 2$ then $X$ and $X / G$ are not $\ka$-rational.

In each of the other cases each of the following options occurs over a suitable field~$\ka$: \mbox{$X$ is $\ka$-rational} and $X / G$ \mbox{is non-$\ka$-rational}, $X$ is non-$\ka$-rational and $X / G$ is non-$\ka$-rational, $X$~is $\ka$-rational and~$X / G$ is $\ka$-rational, $X$~is non-$\ka$-rational and $X / G$ is \mbox{$\ka$-rational}.

\end{theorem}

\begin{proof}

If $G$ is trivial then $X$ and $X / G$ are not $\ka$-rational by Theorem \ref{ratcrit}. The case $d = 2$, $|G| = 2$ is considered in \cite[Remark 4.9]{Tr18b} and \cite[Lemma 5.1]{Tr18b}, and the case $d = 1$, $|G| = 2$ is considered in Remark \ref{DP1typeImin} and \cite[Remark 5.2]{Tr18b}.

For the other cases examples are constructed in \cite[Section 6]{Tr18a}, \cite[Section 6]{Tr16b}, \cite[Section 6]{Tr18b} and Section $5$ below.

\end{proof}

Now we can state the corollaries. We need the following definitions.

\begin{definition}
\label{Eckardtpoint}
A point $p$ on a cubic surface is called \textit{Eckardt point} if there are three $(-1)$-curves passing through $p$.

A point $p$ on a del Pezzo surface of degree $2$ is called \textit{generalized Eckardt point} if there are four $(-1)$-curves passing through $p$.
\end{definition}

\begin{definition}
\label{CubicVIII}
We say that a cubic surface $X$ \textit{has type $\mathrm{VIII}$} if there are three Eckardt points on $\XX$ lying on a line in $\Pro^3_{\ka}$ (this line is not contained in $X$).

Moreover, if these three points are defined over $\ka$ then we say that $X$ \textit{has type $\mathrm{VIII}$-$1$}, and if exactly one of these points is defined over $\ka$ then we say that $X$ \textit{has type $\mathrm{VIII}$-$2$}.
\end{definition}

\begin{remark}
\label{CubicVIIIremark}
Note that for a cubic surface $X$ of type $\mathrm{VIII}$ the group $\Aut\left(\XX\right)$ contains a non-abelian subgroup $\SG_3$ of order $6$ (see \cite[Proposition 9.1.27]{Dol12}). Cubic surfaces with \mbox{$\Aut\left(\XX\right) \cong \SG_3$} are called \textit{cubic surfaces of type $\mathrm{VIII}$} in \cite[Table 9.6]{Dol12}.
\end{remark}

\begin{theorem}
\label{Galunirat}
Let $\ka$ be a field of characteristic zero. Any Galois $\ka$-unirational surface is either $\ka$-rational, or birationally equivalent to a surface admitting a structure of conic bundle, or birationally equivalent to a minimal cubic surface of type $\mathrm{VIII}$. Moreover, this cubic surface has type $\mathrm{VIII}$-$1$ if $\ka$ contains $\omega = e^{\frac{2\pi \ii}{3}}$, and has type $\mathrm{VIII}$-$2$ otherwise. Examples of such cubic surfaces exist for any $\ka$ such that there is at least one non-cube in $\ka$.

\end{theorem}

\begin{remark}
\label{CBunirat}
Note that any del Pezzo surface $X$ of degree $4$, such that $X(\ka) \neq \varnothing$ is Galois $\ka$-unirational by \cite[Theorem~IV.8.1]{Man74}. But a blowup of $X$ at a $\ka$-point not lying on $(-1)$-curves admits a structure of a conic bundle.

Also note that Theorem \ref{Galunirat} does not state that any surface admitting a conic bundle structure is Galois $\ka$-unirational. Moreover, even $\ka$-unirationality of conic bundles is an open question. 

\end{remark}

\begin{remark}
\label{unirationality}
Let us recall some facts about $\ka$-unirationality. Let $\ka$ be a field of characteristic zero, and $X$ be a del Pezzo surface of degree $d$. Then by \cite[Theorem~IV.8.1]{Man74} if~$d = 4$ or $d = 3$ and $X(\ka) \neq \varnothing$ then $X$ is $\ka$-unirational. If $d = 2$ and there is a~$\ka$-point that is not a point of intersection of four $(-1)$-curves and does not lie on the ramification divisor of the anticanonical \mbox{map $X \rightarrow \Pro^2_{\ka}$}, then $X$ is $\ka$-unirational by \cite[Corollary~18]{STV14}. For $d = 1$ unirationality is not known in general case. Moreover, there are no known examples of $\ka$-unirational del Pezzo surfaces of degree $1$ with $\rho(X) = 1$. But if $d = 1$ and~$X$ admits a~structure of conic bundle then $X$ is $\ka$-unirational by \cite[Theorem~7]{KM17}. Moreover, in \cite[Corollary 8]{KM17} it is shown that if a geometrically rational surface $X$ admits a structure of conic bundle $X \rightarrow \Pro^1_{\ka}$ such that $X(\ka) \ne \varnothing$ and $K_X^2 \geqslant 1$, then $X$ is \mbox{$\ka$-unirational} (it seems that conic bundles with $K_X^2 = 4$ that are not del Pezzo surfaces are not considered in \cite{KM17} and the cited there papers, but one can easily deduce $\ka$-unirationality of such surfaces from \cite[Lemma 7.1]{CT88}).
In all these cases $X(\ka)$ is dense.

For some results on unirationality of del Pezzo surfaces in the case $\Char \ka > 0$ see \cite{Kol02}, \cite{Kn15} and \cite{FvL16}.
\end{remark}

The following corollary immediately follows from Theorem \ref{Galunirat}.

\begin{corollary}
\label{Galuniratcor}
Let $\ka$ be a field of characteristic zero, $X$ be a del Pezzo surface of degree~$d$, such that $X(\ka) \neq \varnothing$ and $\rho(X) = 1$. Then~$X$ is Galois $\ka$-unirational if $d \geqslant 4$, and $X$ is not Galois $\ka$-unirational if $d = 1$, $d = 2$ or~$d = 3$ and $X$ is not a cubic surface of type $\mathrm{VIII}$.
\end{corollary}

Therefore any minimal del Pezzo surface $X$ of degree $2$ with dense $X(\ka)$ or a minimal cubic surface $X$ with $X(\ka) \neq \varnothing$ and with less than three Eckardt points gives us an example of $\ka$-unirational but not Galois $\ka$-unirational surface.

Moreover, the following theorem shows that del Pezzo surfaces of degree $1$, del Pezzo surfaces of degree $2$ without generalized Eckardt points defined over $\ka$ and cubic surfaces without Eckardt points defined over $\ka$ are really far from Galois $\ka$-unirationality.

\begin{theorem}
\label{Galgeomunirat}

Any quotient of a geometrically rational surface by a nontrivial finite group is either birationally equivalent to a conic bundle, or birationally equivalent to a del Pezzo surface $S$ with $K_S^2 \geqslant 4$,  or birationally equivalent to a minimal cubic surface with an Eckardt point defined over $\ka$, or birationally equivalent to a minimal del Pezzo surface of degree $2$ with a generalized Eckardt point defined over~$\ka$.

The other minimal del Pezzo surfaces of degree $3$ and $2$, and minimal del Pezzo surfaces of degree $1$ are not birationally equivalent to a quotient of a geometrically rational surface by a nontrivial finite group.

\end{theorem}

\begin{question}
Is it true that any minimal del Pezzo surface of degree $4$ or more without $\ka$-points, or a minimal cubic surface with an Eckardt point defined over $\ka$, or a minimal del Pezzo surface of degree $2$ with a generalized Eckardt point defined over~$\ka$ is birationally equivalent to a quotient of a geometrically rational surface by a nontrivial finite group?
\end{question}

Now we make some general observations about quotients and fields of invariants. As a direct consequence of Theorem \ref{main} we have the following theorem about groups of odd order.

\begin{corollary}[of Theorem \ref{main}]
\label{Unquotient}
Let $\ka$ be a field of characteristic zero, $X$ be a $\ka$-rational surface, \mbox{$G \subset \Aut_{\ka}(X)$} be a finite subgroup of odd order, and $|G| \neq 3$. Then $X / G$ is $\ka$-rational.
\end{corollary}

In terms of the fields of invariants Theorem \ref{Unquotient} can be written in the following way.

\begin{corollary}[of Theorem \ref{main}]
\label{Unfields}
Let $\ka$ be a field of characteristic zero, $G$ be a finite group acting on $K = \ka(x_1, x_2)$ and preserving $\ka$. If~$|G|$ is odd and is not equal to $3$ then $K^G \cong \ka(y_1, y_2)$.
\end{corollary}




The plan of this paper is as follows. In Section $2$ we remind some notions and results about rational surfaces, $G$-equivariant minimal model program and singularities.

In Section $3$ we consider quotients of del Pezzo surfaces of degree $1$. In Proposition~\ref{DP1} we give a list of groups such that the quotient can be non-$\ka$-rational. To prove this proposition we consider groups acting on del Pezzo surfaces of degree $1$, and show that the quotient is $\ka$-rational in all cases that are not listed in Proposition \ref{DP1}. Also in Remark~\ref{DP1typeImin} we show that the quotient $X / G$ of a del Pezzo surface $X$ of degree $1$ by a group~$G$ of order $2$ having an elliptic curve of fixed points is always not $\ka$-rational if~$\rho(X)^G = 1$.

In Section $4$ we study properties of conjugacy classes of elements of order $3$ in the Weyl group $\mathrm{W}(\mathrm{E}_8)$. We introduce a notion of \textit{$\davidsstar$-configuration} of $(-1)$-curves, and show how one can prove $\ka$-rationality, non-$\ka$-rationality and $G$-minimality of del Pezzo surfaces of degree~$1$ in terms of $\davidsstar$-configurations.

In Section $5$ for the cyclic groups of order $3$ and $6$ mentioned in Proposition \ref{DP1} we use obtained results to construct explicit examples of $\ka$-rational and non-$\ka$-rational quotients of $\ka$-rational and non-$\ka$-rational del Pezzo surfaces~$X$ of degree $1$ such that $\rho(X)^G = 1$.

In Section $6$ we remind the construction of a non-$\ka$-rational quotient of a cubic surface by a group of order $3$, and prove Theorems \ref{Galunirat}, \ref{Galgeomunirat} and \ref{Unquotient}.

The author is grateful to Costya Shramov for many useful discussions and comments, to Yuri\,G.\,Prokhorov for posing the problem of the classification of Galois unirational surfaces, and to Katya Kuksa for drawing the figures.

\begin{notation}

Throughout this paper $\ka$ is any field of characteristic zero, $\kka$ is its algebraic closure, $G$ is a finite group. For a surface $X$ we denote $X \otimes \kka$ by $\XX$. For a surface $X$ we denote the Picard group (resp. the $G$-invariant Picard group) by $\Pic(X)$ (resp. $\Pic(X)^G$). The number \mbox{$\rho(X) = \operatorname{rk} \Pic(X)$} (resp. \mbox{$\rho(X)^G = \operatorname{rk} \Pic(X)^G$}) is the Picard number (resp. the $G$-invariant Picard number) of $X$. If two surfaces $X$ and $Y$ are $\ka$-birationally equivalent then we write~$X \approx Y$. If two divisors $A$ and $B$ are linearly equivalent then we write~$A \sim B$.

We use the following notation concerning groups:

\begin{itemize}

\item $\ii = \sqrt{-1}$;

\item $\xi_k = e^{\frac{2\pi \ii}{k}}$;

\item $\omega = \xi_3 = e^{\frac{2\pi \ii}{3}}$;

\item $\CG_n$ denotes a cyclic group of order $n$;

\item $\DG_{2n}$ denotes a dihedral group of order $2n$;

\item $\SG_n$ denotes a symmetric group of degree $n$;

\item $\AG_n$ denotes an alternating group of degree $n$;

\item $\langle g_1, \ldots, g_n \rangle$ denotes a group generated by $g_1$, \ldots, $g_n$;

\item $A \times B$ denotes the direct product of groups $A$ and $B$;


\item $\operatorname{diag}(a, b)$ denotes the diagonal $2 \times 2$ matrix with diagonal entries $a$ and $b$;

\item for a vector space $V$ (or a lattice $L$) with an action of a group $G$ we denote by $V^G$ (resp. $L^G$) the subspace of $G$-invariant vectors (resp. the sublattice of $G$-invariant elements).

\end{itemize}

\end{notation}

\section{Preliminaries}

In this section we review main notions and results of $G$-equivariant minimal model program following the papers \cite{Man67}, \cite{Isk79}, \cite{DI1}.

\begin{definition}
\label{rationality}
A {\it geometrically rational variety} $X$ is a variety over $\ka$ such that $\XX=X \otimes \kka$ is birationally equivalent to $\Pro^n_{\kka}$.

A {\it $\ka$-rational variety} $X$ is a variety over $\ka$ such that $X$ is birationally equivalent to $\Pro^n_{\ka}$.

A variety $X$ over $\ka$ is a {\it $\ka$-unirational variety} if there exists a $\ka$-rational variety $Y$ and a dominant rational map $\varphi: Y \dashrightarrow X$.
\end{definition}

\begin{definition}
\label{minimality}
A {\it $G$-surface} is a pair $(X, G)$ where $X$ is a projective surface over $\ka$ and~$G$ is a finite subgroup of $\Aut_{\ka}(X)$. A morphism of $G$-surfaces $f: X \rightarrow X'$ is called a~\textit{$G$-equivariant morphism} or simply \textit{$G$-morphism} if for each $g \in G$ one has $fg = gf$.

A smooth $G$-surface $X$ is called {\it $G$-minimal} if any birational $G$-equivariant morphism of smooth $X \rightarrow X'$ is an isomorphism.

Let $X$ be a smooth $G$-surface. A $G$-minimal surface $Y$ is called a {\it minimal model} of $(X, G)$ or {\it $G$-minimal model} of $X$ if there exists a birational $G$-equivariant \mbox{morphism $X \rightarrow Y$}.
\end{definition}

The following theorem is a classical result about the $G$-equivariant minimal model program.

\begin{theorem}
\label{GMMP}
Any birational $G$-equivariant morphism $f:X \rightarrow Y$ of smooth $G$-surfaces can be factorized in the following way:
$$
X= X_0 \xrightarrow{f_0} X_1 \xrightarrow{f_1} \ldots \xrightarrow{f_{n-2}} X_{n-1} \xrightarrow{f_{n-1}} X_n = Y,
$$
where each $f_i$ is a contraction of a set $\Sigma_i$ of disjoint $(-1)$-curves on~$X_i$ such that $\Sigma_i$ is defined over $\ka$ and $G$-invariant.
\end{theorem}

The classification of $G$-minimal rational surfaces is well-known due to V.\,Iskovskikh and Yu.\,Manin (see \cite{Isk79} and \cite{Man67}). We introduce some important notions before surveying it.

\begin{definition}
\label{Cbundledef}
A smooth rational $G$-surface $X$ admits a structure of a {\it $G$-equivariant conic bundle} if there exists a $G$-equivariant map $\varphi: X \rightarrow B$ such that any scheme fibre is isomorphic to a reduced conic in~$\Pro^2_{\ka}$ and $B$ is a smooth curve.

The curve $B$ is called the {\it base} of the conic bundle.

\end{definition}

Let $\overline{\varphi}: \XX \rightarrow \overline{B} \cong \Pro^1_{\kka}$ be a conic bundle. A general fibre of $\overline{\varphi}$ is isomorphic to $\Pro^1_{\kka}$. The fibration~$\overline{\varphi}$ has a finite number of singular fibres which are degenerate conics. Any irreducible component of a singular fibre is a $(-1)$-curve. If $n$ is the number of singular fibres of $\overline{\varphi}$ then $K_{\XX}^2 + n = 8$.

\begin{definition}
\label{DPdef}
A {\it del Pezzo surface} is a smooth projective surface~$X$ such that the anticanonical class $-K_X$ is ample.

A {\it singular del Pezzo surface} is a normal projective surface $X$ such that the anticanonical class $-K_X$ is ample and all singularities of $X$ are Du Val singularities.

A {\it weak del Pezzo surface} is a smooth projective surface $X$ such that the anticanonical class $-K_X$ is nef and big.

The number $d = K_X^2$ is called the {\it degree} of a (singular, weak) del Pezzo surface $X$.
\end{definition}

The following proposition is well known (see e.g. \cite[Subsection 8.1.3]{Dol12}). 

\begin{proposition}
\label{DPconnection}

If $\widetilde{X}$ is a weak del Pezzo surface then for any irreducible reduced curve~$C$ one has $C^2 \geqslant -2$.

If $X$ is a singular del Pezzo surface and $\widetilde{X} \rightarrow X$ is the minimal resolution of singularities then $\widetilde{X}$ is a weak del Pezzo surface, and all $(-2)$-curves on $\widetilde{X}$ lie on the preimages of the singularities of $X$.

If $\widetilde{X}$ is a weak del Pezzo surface and $\widetilde{X} \rightarrow Y$ is a birational morphism of smooth surfaces then $Y$ is a weak del Pezzo surface.

If $Y$ is a weak del Pezzo surface and there are no $(-2)$-curves on $Y$ then $Y$ is a del Pezzo surface.

\end{proposition}

We need the following results about classification of $G$-minimal geometrically rational surfaces.

\begin{theorem}[{\cite[Theorem 1]{Isk79}}]
\label{Minclass}
Let $X$ be a $G$-minimal geometrically rational surface. Then either $X$ admits a $G$-equivariant conic bundle structure with $\Pic(X)^{G} \cong \Z^2$, or $X$ is a del Pezzo surface with $\Pic(X)^{G} \cong \Z$.
\end{theorem}

\begin{theorem}[{cf. \cite[Theorems 4 and 5]{Isk79}}]
\label{MinCB}
Let $X$ admit a $G$-equivariant structure of a conic bundle such that $\rho(X)^G = 2$ and $K_X^2 \neq 3$, $5$, $6$, $7$, $8$. Then $X$ is $G$-minimal.
\end{theorem}

At the end of this section we review some results about quotient singularities and their resolutions.

All singularities appearing in this paper are toric singularities. These singularities are locally isomorphic to the quotient of $\A^2$ by the cyclic group generated by $\operatorname{diag}(\xi_m, \xi_m^r)$. Such a singularity is denoted by $\frac{1}{m}(1,r)$. If $\gcd(m,r) > 1$ then the group
$$
\CG_m \cong \langle \operatorname{diag}(\xi_m, \xi_m^r) \rangle
$$
\noindent contains a reflection (i.e. an element with a unique eigenvalue not equal to $1$) and the quotient singularity is isomorphic to a quotient singularity with smaller $m$. A singularity of type $\frac{1}{m}(1, m-1)$ is called \textit{$A_{m-1}$-singularity}. Singularities appearing in this paper have type $A_{m-1}$ or $\frac{1}{3}(1, 1)$.

\begin{remark}
\label{singularities}
Let $P$ be a singular point on a surface $S$ and $\pi: \widetilde{S} \rightarrow S$ be the minimal resolution of this point.

If $P$ is $A_n$-singularity then $K_{\widetilde{S}}^2 = K_S^2$. For a curve $C$ passing through $P$ one has
$$
\pi^{-1}_*(C)^2 - C^2 = -\frac{n}{n+1}.
$$
\noindent The preimage of $P$ on $\widetilde{S}$ is a chain of $n$ $(-2)$-curves.

If $P$ is a singularity of type $\frac{1}{3}(1, 1)$ then $K_{\widetilde{S}}^2 = K_S^2 - \frac{1}{3}$. For a curve $C$ passing through $P$ one has
$$
\pi^{-1}_*(C)^2 - C^2 = -\frac{1}{3}.
$$
\noindent The preimage of $P$ on $\widetilde{S}$ is a $(-3)$-curve.

\end{remark}

\section{Del Pezzo surfaces of degree $1$}

A del Pezzo surface $X$ of degree $1$ is isomorphic to a smooth sextic surface in a \mbox{$\ka$-form} of the weighted projective space $\Pro_{\ka}(1 : 1 : 2 : 3)$. The anticanonical linear system $|-K_X|$ has a unique base point $p$. This point is defined over $\ka$ and is $G$-fixed for any finite group~$G$ acting on $X$.

The linear system $|-2K_X|$ gives a double cover of a quadratic cone in $\Pro^3_{\ka}$ branched in a smooth sextic curve. The corresponding involution $\beta$ on $X$ is called \textit{the Bertini involution}. Obviously the involution $\beta$ commutes with any element of $\Aut_{\ka}(X)$. In particular, $\langle \beta \rangle$ is a normal subgroup in $\Aut_{\ka}(X)$.

One can choose coordinates in $\Pro_{\kka}(1 : 1 : 2 : 3)$ such that the surface $\XX$ is given by the equation
$$
L_6(x, y) + L_4(x, y)z + z^3 + t^2 = 0
$$
\noindent where $L_6$ and $L_4$ are homogeneous polynomials of degrees $6$ and $4$ respectively. The point~$p$ has coordinates $(0 : 0 : -1 : 1)$, and the involution $\beta$ switches the sign of the coordinate $t$. 

Note that, if $L_4(x, y) = 0$ then there is an element $\alpha \in \Aut_{\kka}\left( \XX \right)$ of order $3$, that acts as
$$
(x : y : z : t) \mapsto (x : y : \omega z : t).
$$

For any $G \subset \Aut_{\kka}(\XX)$ the point $p$ is $G$-fixed. Therefore the group $G$ faithfully acts on the tangent space $T_p(\XX)$ of $\XX$ at $p$. In particular, $G$ is a subgroup of $\mathrm{GL}_2(\kka)$, and any element of $G$ can be realized as
$$
(x : y : z : t) \mapsto (l(x, y) : m(x, y) : z : t),
$$
\noindent where $l(x,y)$ and $m(x,y)$ are linear homogeneous polynomials. In particular, one has
$$
\alpha: (x : y : z : t) \mapsto (\omega x : \omega y : z : t), \qquad \beta: (x : y : z : t) \mapsto (-x : -y : z : t).
$$

Let $G'$ be the image of $G$ under the natural map $\mathrm{GL}_2(\kka) \rightarrow \mathrm{PGL}_2(\kka)$, and $K \lhd G$ be the kernel of the map $G \rightarrow G'$.

Any element $g \in K$ acts on $\Pro_{\kka}(1 : 1 : 2 : 3)$ as 
$$
(x : y : z : t) \mapsto \left(\xi_k x : \xi_k y : z : t\right).
$$
\noindent Thus $\xi_k^6 = 1$. Hence $g$ lies in $\langle \alpha, \beta \rangle$. Therefore $K = G \cap \langle \alpha, \beta \rangle$.

\begin{lemma}
\label{DP1alpha}
Let a finite group $G$ act on a del Pezzo surface $X$ of degree $1$ and contain $\beta$ or $\alpha$. Then the quotient $X / G$ is birationally equivalent to a surface $Y$ such that $K_Y^2 \geqslant 5$. Moreover, if $X$ has at least two $\ka$-points then $X / G$ is $\ka$-rational.
\end{lemma}

\begin{proof}
Assume that $G$ contains $\beta$. Then the quotient $X / \langle \beta \rangle$ is a quadratic cone in $\Pro^3_{\ka}$. Let $S \rightarrow X / \langle \beta \rangle$ be the minimal resolution of singularities. Then $S$ is a Hirzebruch surface~$\F_2$, one has $K_S^2 = 8$, and the quotient $X / G \approx S / (G / \langle \beta \rangle)$ is birationally equivalent to a~surface~$Y$ such that $K_Y^2 \geqslant 5$ by Remark \ref{geq5touse}.

Now assume that $G$ contains $\alpha$. The group $N \cong \langle \alpha \rangle$ is a normal subgroup of $G$. The set of $N$-fixed points on $\XX$ consists of a fixed curve $z = 0$, that has class $-2K_X$, and the isolated fixed point $p$. Let~\mbox{$f: X \rightarrow X / N$} be the quotient morphism. By the Hurwitz formula (see \cite[Equation (1.38)]{Dol12}) one has $K_X = f^*(K_{X / N}) - 4K_X$. Therefore
$$
K_{X / N}^2 = \frac{1}{3}(5K_X)^2 = \frac{25}{3},
$$
and the surface $X / N$ has a unique singular point of type $\frac{1}{3}(1,1)$. Let $\pi: \widetilde{X / N} \rightarrow X / N$ be the minimal resolution of singularities. From Remark \ref{singularities} one can see that
$$
K_{\widetilde{X / N}}^2 = K_{X / N}^2 - \frac{1}{3} = 8.
$$ 
Therefore the quotient $X / G \approx \widetilde{X / N} / (G / N)$ is birationally equivalent to a surface $Y$ such that $K_Y^2 \geqslant 5$ by Remark \ref{geq5touse}.

If $X$ has at least two $\ka$-points then one of these points $q$ differs from $p$. Then the image of $q$ is a smooth $\ka$-point on $X / \langle \beta \rangle$ or $X / \langle \alpha \rangle$. Thus these quotients are $\ka$-rational by Theorem \ref{ratcrit}. Therefore the set of $\ka$-points on $X / G$ is dense, and $X / G \approx Y$ is $\ka$-rational by Theorem \ref{ratcrit}.
\end{proof}

\begin{corollary}
\label{DP1bound}
Assume that a finite group $G$ acts on a del Pezzo surface $X$ of degree~$1$, and the quotient $X / G$ is not birationally equivalent to a surface $Y$ such that $K_Y^2 \geqslant 5$. Then $G$ is either cyclic or dihedral of order $4k + 2$.
\end{corollary}

\begin{proof}
By Lemma \ref{DP1alpha} the group $G$ does not contain elements $\beta$ and $\alpha$. Therefore $G$ is a~subgroup of $\mathrm{GL}_2(\kka)$ and $\mathrm{PGL}_2(\kka)$, since it faithfully acts on the tangent space $T_p(\XX)$ and its projectivization. It is well known (see \cite[Chapter III]{Bl17}) that any finite subgroup of~$\mathrm{PGL}_2(\kka)$ is isomorphic to $\CG_n$, $\DG_{2n}$, $\AG_4$, $\SG_4$ or~$\AG_5$. Moreover, groups $\DG_{4k}$, $\AG_4$, $\SG_4$ and~$\AG_5$ do not have isomorphic lift to $\mathrm{GL}_2(\kka)$. Thus $G$ is either cyclic or dihedral of order~$4k + 2$.
\end{proof}

From Corollary \ref{DP1bound} and the results of \cite[Subsection 8.8.4]{Dol12} we can obtain the following theorem.

\begin{theorem}[{cf. \cite[Subsection 8.8.4]{Dol12}}]
\label{DP1groupclass}
Let $\XX$ be a del Pezzo surface of degree~$1$, and $G$ be a subgroup of $\Aut_{\kka}(\XX)$ not containing the elements $\beta$ and $\alpha$. Then for each possibility of $G$ one can choose coordinates in $\Pro_{\kka}(1 : 1 : 2 : 3)$ such that the equation of $\XX$ and the group $G$ are presented in Table \ref{table1}.

\vbox{
\begin{center}
\begin{longtable}{|c|c|c|c|}
\caption[]{\label{table1} Automorphism groups}\endhead
\hline
Type & Group & Equation & Action \\
\hline
$0$ & $\langle \mathrm{id} \rangle$ & $L_6(x, y) + L_4(x, y)z + z^3 + t^2 = 0$ & trivial \\
\hline
$\mathrm{I}$ & $\CG_2$ & $L_3\left(x^2, y^2\right) + L_2\left(x^2, y^2\right)z + z^3 + t^2 = 0$ & $(-x : y : z : t)$ \\
\hline
$\mathrm{II}$ & $\CG_3$ & $x^6 + Ax^3y^3 + y^6 + \left(Bx^3 + Cy^3\right)yz + z^3 + t^2 = 0$ & $(\omega x : y : z : t)$ \\
\hline
$\mathrm{III}$ & $\CG_3$ & $x^6 + Ax^3y^3 + y^6 + Bx^2y^2z + z^3 + t^2 = 0$ & $(\omega x : \omega^2 y : z : t)$ \\
\hline
$\mathrm{IV}$ & $\CG_4$ & $\left(Ax^4 + By^4\right)y^2 + \left(Cx^4 + Dy^4\right)z + z^3 + t^2 = 0$ & $(\ii x : y : z : t)$ \\
\hline
$\mathrm{V}$ & $\CG_4$ & $\left(Ax^4 + By^4\right)y^2 + \left(Cx^4 + Dy^4\right)z + z^3 + t^2 = 0$ & $(\ii x : -y : z : t)$ \\
\hline
$\mathrm{VI}$ & $\CG_5$ & $\left(x^5 + y^5\right)y + Ay^4z + z^3 + t^2 = 0$ & $(\xi_5 x : y : z : t)$ \\
\hline
$\mathrm{VII}$ & $\CG_6$ & $x^6 + y^6 + Ay^4z + z^3 + t^2 = 0$ & $(\xi_6 x : y : z : t)$ \\
\hline
$\mathrm{VIII}$ & $\CG_6$ & $x^6 + y^6 + Ay^4z + z^3 + t^2 = 0$ & $(\omega x : -y : z : t)$ \\
\hline
$\mathrm{IX}$ & $\CG_6$ & $x^6 + y^6 + Bx^2y^2z + z^3 + t^2 = 0$ & $(\xi_6 x : \omega y : z : t)$ \\
\hline
$\mathrm{X}$ & $\CG_{12}$ & $x^6 + y^4z + z^3 + t^2 = 0$ & $(\omega x : \ii y : z : t)$ \\
\hline
$\mathrm{XI}$ & $\CG_{12}$ & $x^6 + y^4z + z^3 + t^2 = 0$ & $(\xi_6 x : \ii y : z : t)$ \\
\hline
$\mathrm{XII}$ & $\DG_6$ & $x^6 + Ax^3y^3 + y^6 + Bx^2y^2z + z^3 + t^2 = 0$ & $(\omega x : \omega^2 y : z : t)$, $(y : x : z : t)$ \\
\hline

\end{longtable}
\end{center}}

\noindent where $L_k$ is a homogeneous polynomial of degree $k$, and $A$, $B$, $C$ and $D$ are elements of~$\kka$.

\end{theorem}

For shortness, we will say that a group $G$ has type $\mathrm{I}$, type $\mathrm{II}$, etc. of Theorem \ref{DP1groupclass}, if this group acts on a surface $\XX$ of the corresponding type.

For completeness we give a sketch of proof of Theorem \ref{DP1groupclass}. 

\begin{proof}[Proof of Theorem \ref{DP1groupclass}]

We start from the case of cyclic groups. Let $G$ be a cyclic group generated by an element $g$. The action of $g$ can be diagonalized in the following way:
$$
(x: y: z: t) \mapsto (\lambda x : \mu y : z : t).
$$

The group $G$ does not contain $\alpha$ and $\beta$, therefore $G$ faithfully acts on $\Pro^1_{\kka}$ with coordinates $x$ and $y$. Moreover, the group $G$ should act on the roots of the polynomials $L_6(x,y)$ and $L_4(x,y)$. The only two $G$-fixed points on $\Pro^1_{\kka}$ are $(1 : 0)$ and $(0 : 1)$, and the other points are contained in orbits of length $\ord g$. Thus if $\ord g > 6$ then $L_6(x, y)$ is a monomial, and $L_4(x, y)$ is a monomial or is equal to $0$.

Assume that $\ord g > 6$. If $L_4(x, y)$ is equal to zero or proportional to $x^3y$, $x^2y^2$ or~$xy^3$ then $x = 0$ or $y = 0$ is a double root of $L_6(x,y)$ that is a common root with~$L_4(x,y)$. In these cases $X$ is singular. Therefore without loss of generality we may assume that~$L_4(x, y) = y^4$, and $L_6(x, y) = x^6$ or $L_6(x, y) = x^5y$.

In the case $L_6(x, y) = x^6$ one has $\lambda^6 = \mu^4 = 1$. If $\mu^2 = 1$ then $\ord g \leqslant 6$, and otherwise we can put $\mu = \ii$. One has $\ord g > 6$ if $\lambda$ is $\xi_6$, $\omega$, $\omega^2$ or $\xi_6^5$. For each of those possibilities the element $g$ is a generator of a group of type $\mathrm{X}$ or $\mathrm{XI}$.

In the case $L_5(x, y) = x^5y$ one can consider a subgroup of order $2$ in $G$, and see that it is generated by $\beta$. Therefore we have a contradiction.

Now assume that $\ord g \leqslant 6$. If $\ord g = 2$ then without loss of generality $\lambda = -1$. If~$\mu = 1$ then $G$ has type $\mathrm{I}$, and if $\mu = -1$ then $g = \beta$ and we have a contradiction.

If $\ord g = 3$ then without loss of generality $\lambda = \omega$. If~$\mu = 1$ then $G$ has type $\mathrm{II}$, if $\mu = \omega$ then $g = \alpha$ and we have a contradiction, and if $\mu = \omega^2$ then $G$ has type $\mathrm{III}$.

If $\ord g = 4$ then without loss of generality $\lambda = \ii$. If~$\mu = 1$ then $G$ has type $\mathrm{IV}$, if~$\mu = \pm \ii$ then $g^2 = \beta$ and we have a contradiction, and if $\mu = -1$ then $G$ has type $\mathrm{V}$.

If $\ord g = 5$ then without loss of generality $\lambda = \xi_5$. If~$\mu = 1$ then $G$ has type $\mathrm{VI}$, and if $\mu = \xi_5$ then $L_6(x,y)$ is not $G$-invariant. If $\mu = \xi_5^2$ then $L_6(x, y) = Ax^2y^4$, \mbox{$L_4(x,y) = Bx^3y$} and $x = 0$ is a common double root of $L_4(x,y)$ and $L_6(x,y)$. In this case $X$ is singular. The case $\mu = \xi_5^3$ is analogous, since $g^3$ has eigenvalues $\xi_5$ and $\xi_5^2$. If $\mu = \xi_5^4$ then \mbox{$L_6(x, y) = Ax^3y^3$}, $L_4(x,y) = Bx^2y^2$, $x = 0$ is a common double root of $L_4(x,y)$ and~$L_6(x,y)$, and $X$ is singular.

If $\ord g = 6$ then either $G$ has type $\mathrm{VIII}$, or without loss of generality $\lambda = \xi_6$. If~$\mu = 1$ then $G$ has type $\mathrm{VII}$, if~$\mu$ equals $\xi_6$, $-1$ or $\xi_6^5$ then $g^3 = \beta$ and we have a contradiction, if $\mu = \omega$ then $G$ has type $\mathrm{IX}$, and if~$\mu=\omega^2$ then $g^2 = \alpha$ and we have a contradiction.

If $G$ is noncyclic then by Corollary \ref{DP1bound} the group $G$ is dihedral of order $4k + 2$. Such a group contains a normal cyclic subgroup $N$ of order $2k + 1$ generated by an element $g$. Moreover, the eigenvalues of $g$, considered as an element in $\mathrm{GL}_2 \left( \kka \right)$, should be inverse. Therefore the only possibility for $N$ is type $\mathrm{III}$, and $G$ has type $\mathrm{XII}$.

\end{proof}

In this section we prove the following proposition.

\begin{proposition}
\label{DP1}
Let $X$ be a del Pezzo surface of degree $1$, and $G$ be a finite subgroup of~$\Aut_{\ka}(X)$ such that there are at least two $\ka$-points on $X$. Then $X / G$ can be \mbox{non-$\ka$-rational} only if the group $G$ has type $0$, $\mathrm{I}$, $\mathrm{II}$, $\mathrm{III}$, $\mathrm{IX}$ of Theorem \ref{DP1groupclass}. For any other group~$G$ the quotient $X / G$ is $\ka$-rational.
\end{proposition}

We start from groups containing a normal subgroup that has a curve consisting of fixed points.

\begin{lemma}
\label{DP1typeI}
Let a finite group $G$ act on a del Pezzo surface $X$ of degree $1$, and $N \cong \CG_2$ be a normal group of type $\mathrm{I}$ in $G$. Then the quotient $X / N$ is $G / N$-birationally equivalent to a del Pezzo surface of degree $3$. Moreover, the set of $\ka$-points on $X / N$ is dense.
\end{lemma}

\begin{proof}

By Theorem \ref{DP1groupclass} one can choose coordinates in $\Pro_{\kka}(1 : 1 : 2 : 3)$ such that the set of $N$-fixed points on $\XX$ consists of a fixed curve $x = 0$, that has class $-K_X$, and three isolated fixed $p_1$, $p_2$ and $p_3$, lying on the line $y = t = 0$. The curve $C$ given by $y = 0$ is~$N$-invariant and passes through the points $p_i$.

Let $f: X \rightarrow X / N$ be the quotient morphism, and
$$
\pi: \widetilde{X / N} \rightarrow X / N
$$
\noindent be the minimal resolution of singularities. The points $f(p_i)$ are $A_1$-singularities on $X / N$, and the curve $f(C)$ passes through these three points. Therefore
$$
\pi^{-1}_*f(C)^2 = f(C)^2 - 3 \cdot \frac{1}{2} = \frac{1}{2} C^2 - \frac{3}{2} = -1.
$$
Let $\widetilde{X / N} \rightarrow Y$ be the contraction of the $(-1)$-curve $\pi^{-1}_*f(C)$. The surface $X / N$ is a~singular del Pezzo surface. Therefore $\widetilde{X / N}$ is a weak del Pezzo surface. There are no negative curves with self-intersection less than $-1$ on $Y$, thus the surface $Y$ is a del Pezzo surface by Proposition \ref{DPconnection}. Its degree is equal to
$$
K_Y^2 = K_{\widetilde{X / N}}^2 + 1 = K_{X / N}^2 + 1 = \frac{1}{2} (2K_X)^2 + 1 = 3.
$$

The cubic surface $Y$ contains a $\ka$-point that is the image of the $(-1)$-curve $\pi^{-1}_*f(C)$. Therefore $Y$ is $\ka$-unirational by Remark \ref{unirationality}, and the sets of $\ka$-points on $Y$ and $X / N$ are dense.
\end{proof}

\begin{remark}
\label{DP1typeImin}
Assume that $G = N$, $\rho(X)^G = 1$ and let us show that under the assumptions of Lemma \ref{DP1typeI} the surface $Y$ is not $\ka$-rational. Let $S_1$, $S_2$ and $S_3$ be the proper transforms of the curves $\pi^{-1}_*f(p_i)$ on $Y$. One has $S_i^2 = -1$.

Assume that all three points $p_i$ are defined over $\ka$. Then each curve $S_i$ is defined over~$\ka$, and $\rho(Y) = 3$. The complete linear system $|S_1 + S_2|$ gives a conic bundle \mbox{structure~$\varphi: Y \rightarrow \Pro^1_{\ka}$}. One can contract the curve $S_1$, that is a component of a singular fibre of~$\varphi$, and get a conic bundle $Z \rightarrow \Pro^1_{\ka}$ such that $K_Z^2 = 4$ and $\rho(Z) = 2$. The~surface~$Z$ is minimal by Theorem \ref{MinCB}, and is not $\ka$-rational by Theorem \ref{ratcrit}.

If there is a unique point $p_j$ defined over $\ka$, then one can contract the $(-1)$-curve $S_j$ and get a minimal del Pezzo surface $Z$ of degree $4$ with $\rho(Z) = 1$. The surface $Z$ is not $\ka$-rational by Theorem \ref{ratcrit}.

If the points $p_i$ are transitevely permuted by the Galois group $\Gal\left(\kka / \ka \right)$, then $\rho(Y) = 1$, and $Y$ is not $\ka$-rational by Theorem \ref{ratcrit}.

\end{remark}

\begin{lemma}
\label{DP1typeInorm}
Let a finite group $G$ of type $\mathrm{IV}$, $\mathrm{V}$, $\mathrm{VII}$, $\mathrm{VIII}$, $\mathrm{X}$ or $\mathrm{XI}$ act on a del Pezzo surface $X$ of degree $1$. Then the quotient $X / G$ is $\ka$-rational.
\end{lemma}

\begin{proof}
The group $G$ contains a normal subgroup $N \cong \CG_2$ of type $\mathrm{I}$. Therefore by~Lemma~\ref{DP1typeI} the quotient $X / N$ is $G / N$-birationally equivalent to a cubic surface $Y$, and $Y(\ka)$ is dense.

One has $X / G \approx Y / (G / N)$. Note that if $G$ has type $\mathrm{VII}$ or $\mathrm{VIII}$ then \mbox{the group~$G / N \cong \CG_3$} has a curve consisting of fixed points on $Y$, and for the other types the order of $G / N$ is not $3$. Therefore in all cases the quotient~$Y / (G / N)$ is $\ka$-rational by~\cite[Theorem 1.3]{Tr16b}.

\end{proof}

\begin{lemma}
\label{DP1typeII}
Let a finite group $G$ act on a del Pezzo surface $X$ of degree $1$, and $N \cong \CG_3$ be a normal subgroup of type $\mathrm{II}$ in $G$. Then the quotient $X / N$ is $G / N$-birationally equivalent to a surface admitting a structure of conic bundle with $4$ singular fibres. Moreover, the set of $\ka$-points on $X / N$ is dense.
\end{lemma}

\begin{proof}
By Theorem \ref{DP1groupclass} one can choose coordinates in $\Pro_{\kka}(1 : 1 : 2 : 3)$ such that the set of~$N$-fixed points on $\XX$ consists of a fixed curve $x = 0$, that has class $-K_X$, and two isolated fixed points $p_1$ and $p_2$ having coordinates $(1 : 0 : 0 :\pm \ii)$. The curve $C$ given by~$y = 0$ is~$N$-invariant and passes through the points $p_i$.

Let $f: X \rightarrow X / N$ be the quotient morphism, and
$$
\pi: \widetilde{X / N} \rightarrow X / N
$$
\noindent be the minimal resolution of singularities. The points $f(p_i)$ are $A_2$-singularities on $X / N$, and the curve $f(C)$ passes through these two points. Therefore
$$
\pi^{-1}_*f(C)^2 = f(C)^2 - \frac{4}{3} = \frac{1}{3}C^2 - \frac{4}{3} = -1.
$$
\noindent Therefore one can $G / N$-equivariantly contract $\pi^{-1}_*f(C)$ and get a surface $Y$. Its degree is equal to
$$
K_Y^2 = K_{\widetilde{X / N}}^2 + 1 = K_{X / N}^2 + 1 = \frac{1}{3} (3K_X)^2 + 1 = 4.
$$

Let $\mathcal{L}$ be a linear system on $\XX$ given by $\lambda y^2 = \mu z$. A general member of $\mathcal{L}$ is an \mbox{$N$-invariant} curve of genus $2$ passing through the isolated fixed points of $N$. Therefore a general member of the linear system $f_*(\mathcal{L})$ is a conic passing through the two singular points on~$X / N$. Hence the $(-1)$-curve $\pi^{-1}_*f(C)$ is contained in the unique member of the linear system $\pi^{-1}_*f_*(\mathcal{L})$ that is a chain of three genus $0$ curves, and the other members of this linear system are conics. Therefore $Y$ admits a structure of a conic bundle. This bundle has $4$ singular fibres since $K_Y^2 = 4$.

The image of $\pi^{-1}_*f(C)$ is a $\ka$-point on $Y$. Therefore $Y$ is $\ka$-unirational by Remark \ref{unirationality}, and the sets of $\ka$-points on $Y$ and $X / N$ are dense.
\end{proof}

\begin{remark}
\label{DP1typeIImin}
Note that by Theorems \ref{Minclass} and \ref{MinCB} the obtained surface $Y$ is $G / N$-minimal, if and only if $\rho(X)^G = 1$ and the two isolated fixed points of $N$ are permuted by the \mbox{group $G \times \Gal \left( \kka / \ka \right)$}, since otherwise $\rho\left( Y \right)^{G / N} > 2$.

In particular, if $G = N$ and $\rho(X)^N = 1$, then by Theorem \ref{ratcrit} the quotient $X / N$ is $\ka$-rational if and only if the two isolated fixed points of $N$ are not permuted by the group~$\Gal \left( \kka / \ka \right)$.
\end{remark}

\begin{lemma}
\label{DP1typeVI}
Let a finite group $G$ of type $\mathrm{VI}$ act on a del Pezzo surface $X$ of degree $1$. Then the quotient $X / G$ is $\ka$-rational.

\end{lemma}

\begin{proof}

By Theorem \ref{DP1groupclass} one can choose coordinates in $\Pro_{\kka}(1 : 1 : 2 : 3)$ such that the set of~$G$-fixed points on $\XX$ consists of a fixed curve $x = 0$, that has class $-K_X$, and an isolated fixed point $q = (1 : 0 : 0 : 0)$.

Let $f: X \rightarrow X / G$ be the quotient morphism, and $Y \rightarrow X / G$ be the minimal resolution of singularities. By the Hurwitz formula one has
$$
K_{X / G}^2 = \frac{1}{5}(5K_X)^2 = 5.
$$

The point $f(q)$ is $A_4$-singularity on $X / G$. Therefore $K_Y^2 = K_{X / G}^2 = 5$. Moreover, at least one point of intersection of $(-2)$-curves on $Y$ is defined over $\ka$. Thus the assertion of the lemma follows from Theorem \ref{ratcrit}. 

\end{proof}

Now we consider groups containing a normal subgroup of type $\mathrm{III}$.

\begin{lemma}
\label{DP1typeIII}

Let a finite group $G$ act on a del Pezzo surface $X$ of degree $1$, and $N \cong \CG_3$ be a normal subgroup of type $\mathrm{III}$ in $G$. Then the quotient $X / N$ is $G / N$-birationally equivalent to a surface $Y$ such that $K_Y^2 = 1$ and $Y$ admits structure of a conic bundle. Moreover, the set of $\ka$-points on $X / N$ is dense.

\end{lemma}

\begin{proof}

The group $N$ has five isolated fixed points $p = (0: 0 : -1 : 1)$ and
$$
q_{x1} = (0 : 1 : 0 : \ii), \quad q_{x2} = (0 : 1 : 0 : -\ii), \quad q_{y1} = (1 : 0 : 0 : \ii), \quad q_{y2} = (1 : 0 : 0 : -\ii).
$$
On the tangent space of $\XX$ at the point $p$ the group $N$ acts as $\langle\operatorname{diag}(\omega, \omega^2)\rangle$ and on the tangent spaces of $\XX$ at the points $q_{xi}$ and $q_{yi}$ the group $N$ acts as~$\langle\operatorname{diag}(\omega, \omega)\rangle$.

Let $C_x$ and $C_y$ be curves given by $x = 0$ and $y = 0$ respectively. These curves \mbox{are $N$-invariant}. The curve $C_x$ passes through the points $p$ and $q_{xi}$, and the curve $C_y$ passes through the points $p$ and $q_{yi}$.

The surface $X / N$ has an $A_2$-singularity and four $\frac{1}{3}(1,1)$-singularities. Let \mbox{$f: X \rightarrow X / N$} be the quotient morphism, and
$$
\pi: \widetilde{X / N} \rightarrow X / N
$$
\noindent be the minimal resolution of singularities. One can easily check that the proper transforms $\pi^{-1}_*f(C_x)$ and $\pi^{-1}_*f(C_y)$ are two disjoint $(-1)$-curves. Let $h: \widetilde{X / N} \rightarrow Y$ be the \mbox{$G / N$-equivariant} contraction of these curves. Then
$$
K_Y^2 = K_{\widetilde{X / N}}^2 + 2 = K_{X/ N}^2 - 4 \cdot \frac{1}{3} + 2 = \frac{1}{3}K_X^2 + \frac{2}{3} = 1.
$$

Let $\mathcal{L}$ be a linear system on $\XX$ given by $\lambda xy = \mu z$. A general member of $\mathcal{L}$ is \mbox{an $N$-invariant} curve of genus $2$ passing through the $N$-fixed points $q_{xi}$ and $q_{yi}$. Therefore a general member of the linear system $f_*(\mathcal{L})$ is a conic passing through the \mbox{four $\frac{1}{3}(1,1)$-singularities} on~$X / N$. Hence the $(-1)$-curves $\pi^{-1}_*f(C_x)$ and $\pi^{-1}_*f(C_y)$ are contained in the unique member $\widetilde{S}$ of the linear system $\pi^{-1}_*f_*(\mathcal{L})$ that is a chain of four genus $0$ curves, and the other members of this linear system are conics. Therefore the linear system $h_*\pi^{-1}_*f_*(\mathcal{L})$ defines a structure of a conic bundle on~$Y$.

The point of intersection of the irreducible components of the singular fibre $S = h \left( \widetilde{S} \right)$ is a $\ka$-point on $Y$. Therefore $Y$ is $\ka$-unirational by Remark \ref{unirationality}, and the sets of $\ka$-points on $Y$ and $X / N$ are dense.

\end{proof}

\begin{remark}
\label{DP1typeIIImin}

Note that if at least one of the points $q_{xi}$ or $q_{yi}$ is defined over $\ka$ then the constructed conic bundle $Y \rightarrow \Pro^1_{\ka}$ has a section defined over $\ka$. Therefore it is $\ka$-rational.

If $\rho(X)^G = 1$, the points $q_{xi}$ are permuted by the group $G \times \Gal \left( \kka / \ka \right)$, and the points~$q_{yi}$ are permuted by this group, then one has
$$
\rho(Y)^{G / N} = \rho\left(\widetilde{X / N}\right)^{G / N} - 2 = \rho(X / N)^{G / N} + 2 = \rho(X)^G + 2 = 3.
$$
Moreover, there is a singular fibre $S$ which components are $G / N$-invariant and defined over $\ka$. Let $Y \rightarrow Z$ be a $G / N$-equivariant contraction of one of these components. Then $K_Z^2 = 2$, $\rho(Z)^{G / N} = 2$, and $Z$ is $G / N$-minimal by Theorem \ref{MinCB}.

Also, if $\rho(X)^G = 1$, and all the points $q_{xi}$ and $q_{yi}$ lie in one $G \times \Gal \left( \kka / \ka \right)$-orbit, then one can check that $\rho(Y)^{G / N} = 2$, and $Y$ is $G / N$-minimal by Theorem \ref{MinCB}.

In particular, if in these cases $G = N$, then by Theorem \ref{ratcrit} the quotient $X / N$ is \mbox{not~$\ka$-rational}.
\end{remark}

\begin{lemma}
\label{DP1typeIX}

Let a finite group $G \cong \CG_6$ of type $\mathrm{IX}$ act on a del Pezzo surface $X$ of~degree~$1$. Then the quotient $X / G$ is birationally equivalent to a surface $W$ such that~$K_W^2 = 4$ and $W$ admits structure of a conic bundle. Moreover, the set of $\ka$-points on $X / G$ is dense.

\end{lemma}

\begin{proof}
A group of type $\mathrm{IX}$ contains a normal subgroup $N$ of type $\mathrm{III}$. By Lemma \ref{DP1typeIII} the quotient $X / N$ is $G / N$-birationally equivalent to a surface~$Y$ such that $K_Y^2 = 1$ and~$Y$ admits a structure of a conic bundle $Y \rightarrow B$. The group $G / N \cong \CG_2$ faithfully acts on~$\mathcal{L}$. Therefore this group faithfully acts on the base of the conic bundle $Y \rightarrow B$.

The group $G / N$ has two invariant fibres of $Y \rightarrow B$. The conic bundle $Y \rightarrow B$ has seven singular geometric fibres, therefore one of the $G / N$-invariant fibres $S$ is singular, and the other $G / N$-invariant fibre $F$ is smooth.

Let $f: Y \rightarrow Y / (G / N)$ be the quotient morphism and
$$
\pi: \widetilde{Y / (G / N)} \rightarrow Y / (G / N)
$$
\noindent be the minimal resolution of singularities. All fibres of the composition of morphisms
$$
\widetilde{Y / (G / N)} \rightarrow Y / (G / N) \rightarrow B / (G / N)
$$
\noindent except $\pi^{-1}f(F)$ and $\pi^{-1}f(S)$ are conics.

The group $G / N$ has two isolated fixed points $r_1$ and $r_2$ on the fibre $F \cong \Pro^1_{\ka}$. Therefore there are two $A_1$-singularities on $f(F)$, and $\pi^{-1}f(F)$ is a chain of three curves with self-intersection numbers $-2$, $-1$ and $-2$. One can contract the $(-1)$-curve and get a singular fibre.

One can check that one component of $S$ is pointwisely fixed by $G / N$, and on the other~$G / N$ acts faithfully. Therefore $\pi^{-1}f(S)$ is a chain of three curves with self-intersection numbers $-2$, $-1$ and $-2$. Moreover, the $(-2)$-curves cannot be permuted by the Galois group, since exactly one of these curves is the preimage of $A_1$-singularity. Thus in the fibre $\pi^{-1}f(S)$ one can consequently contract the $(-1)$-curve and one of the transforms of the $(-2)$-curves, and get a smooth fibre.

As a result of these operations we get a conic bundle $W \rightarrow \Pro^1_{\ka}$ with four singular fibres: one is the transform of $F$, and the three others are images of $G/N$-invariant pairs of singular fibres of $Y \rightarrow B$. One has $K_W^2 = 4$, and $W \approx Y / (G / N) \approx X / G$.

The sets of $\ka$-points on $W$, $Y / (G / N)$ and $X / G$ are dense, since the set of $\ka$-points on~$Y$ is dense by Lemma \ref{DP1typeIII}.

\end{proof}

\begin{remark}
\label{DP1typeIXmin}
Note that the $G / N$-invariant fibres $S$ and $F$ of $Y \rightarrow B$ correspond to the members $xy = 0$ and $z = 0$ of the linear system $\mathcal{L}$ constructed in the proof of Lemma~\ref{DP1typeIII}. The group $\CG_2 \subset G$ has two fixed points on the curve $z = 0$:
$$
q_{x1} = (0 : 1 : 0 : \ii), \quad q_{x2} = (0 : 1 : 0 : -\ii).
$$

Note that if these points are defined over $\ka$ then the isolated $G/N$-fixed points $r_1$ and~$r_2$ on $Y$ are defined over $\ka$. Thus $W$ is not minimal, and for its minimal model $U$ one has~$K_U^2 \geqslant K_W^2 + 1 = 5$. Therefore $W$ is $\ka$-rational by Theorem \ref{ratcrit}.

If $\rho(X)^G = 1$ and the points $q_{xi}$ are permuted by the group $\Gal \left( \kka / \ka \right)$, then $\rho(Y)^{G / N} = 3$ by Remark \ref{DP1typeIIImin}, since the points~$q_{yi}$ are permuted by the group $\CG_2 \subset G$. Moreover, the points $r_1$ and $r_2$ are permuted by the group $\Gal \left( \kka / \ka \right)$. So one has
$$
\rho(W) = \rho\left( \widetilde{Y / (G / N)} \right) - 3 = \rho(Y / (G / N)) - 1 = \rho(Y)^{G / N} - 1 = 2.
$$
Therefore $W$ is minimal by Theorem \ref{MinCB}, and the quotient \mbox{$X / G \approx Y / (G / N) \approx W$} is not $\ka$-rational by Theorem \ref{ratcrit}.
\end{remark}

\begin{lemma}
\label{DP1typeXII}
Let a finite group $G$ of type $\mathrm{XII}$ act on a del Pezzo surface $X$ of degree~$1$. Then the quotient $X / G$ is $\ka$-rational.
\end{lemma}

\begin{proof}

A group of type $\mathrm{XII}$ contains a normal subgroup $N$ of type $\mathrm{III}$. By Lemma \ref{DP1typeIII} the quotient $X / N$ is $G / N$-birationally equivalent to a surface~$Y$ such that $K_Y^2 = 1$ and~$Y$ admits a structure of a conic bundle $Y \rightarrow B$. The group $G / N \cong \CG_2$ trivially acts on~$\mathcal{L}$, since any element of $G \cong \DG_6$ preserves any curve in $\mathcal{L}$. Therefore this group trivially acts on the base of the conic bundle $Y \rightarrow B$.

The group $G / N$ faithfully acts on any smooth fibre of $Y \rightarrow B$. Therefore this group fixes two points on any smooth fibre, and the set of $(G / N)$-fixed points on $Y$ consists of a bisection $C$ of $Y \rightarrow B$, and some isolated fixed points lying on singular fibres.

Let $f: Y \rightarrow Y / (G / N)$ be the quotient morphism and
$$
\pi: \widetilde{Y / (G / N)} \rightarrow Y / (G / N)
$$
\noindent be the minimal resolution of singularities. Consider a singular fibre $F$ of the conic bundle~$Y \rightarrow B$ consisting of two irreducible components $E$ and $E'$. There are two possibilities: either the group $G / N \cong \CG_2$ permutes $E$ and $E'$, or this group preserves the components. In the former case there are no isolated fixed points on $F$, and $f(F)$ is a smooth conic. In~the latter case there is a unique isolated fixed point $E \cap E'$ of $G / N$. Therefore $f(E \cap E')$ is $A_1$-singularity, and $\pi^{-1}f(F)$ is a chain of three curves with self-intersection numbers~$-1$, $-2$ and $-1$. For each such chain one can contract two $(-1)$-curves, and get a conic bundle~$Z \rightarrow B$ without singular fibres. One has $K_Z^2 = 8$, and $X / G \approx Y / (G / N) \approx Z$.

The sets of $\ka$-points on $Z$, $Y / (G / N)$ and $X / G$ are dense, since the set of $\ka$-points on~$Y$ is dense by Lemma \ref{DP1typeIII}. Therefore $X / G \approx Y / (G / N) \approx Z$ is $\ka$-rational by Theorem \ref{ratcrit}.


\end{proof}

Now we can prove Proposition \ref{DP1}.

\begin{proof}[Proof of Proposition \ref{DP1}]
If $G$ contains an element $\alpha$ or $\beta$ then the quotient $X / G$ is \mbox{$\ka$-rational} by Lemma \ref{DP1alpha}.

Otherwise, $G$ is a group listed in Table \ref{table1}. If $G$ is a group that is not listed in Proposition~\ref{DP1} then $X / G$ is $\ka$-rational by Lemmas \ref{DP1typeInorm}, \ref{DP1typeVI} and \ref{DP1typeXII}.

\end{proof}

\section{Elements of order $3$ in $\mathrm{W}(\mathrm{E}_8$)}

A del Pezzo surface $\XX$ of degree $1$ over an algebraically closed field $\kka$ is isomorphic to a~blowup of~$\Pro^2_{\kka}$ at eight points $p_1$, $\ldots$, $p_8$ in general position. Therefore the group~$\Pic(\XX)$ is generated by the proper transform $L$ of the class of a line on $\Pro^2_{\kka}$, and the classes \mbox{$E_1$, $\ldots$, $E_8$} of the exceptional divisors. The sublattice $K_X^{\perp}$ of classes $C$ in $\Pic(\XX)$ such that $C \cdot K_X = 0$, is generated by 
$$
L - E_1 - E_2 - E_3, \qquad E_1 - E_2, \qquad E_2 - E_3, \qquad \ldots, \qquad E_7 - E_8.
$$
This set of generators are simple roots for the root system of type $\mathrm{E}_8$. Therefore any group acting on the Picard lattice $\Pic(\XX)$ and preserving the intersection form is a subgroup of the Weyl group $\mathrm{W}(\mathrm{E}_8)$. Moreover, if a finite group $G$ acts on a del Pezzo surface~$X$ of degree $1$ then there is an embedding $G \hookrightarrow \mathrm{W}(\mathrm{E}_8)$ (see \cite[Lemma 6.2]{DI1}). For~convenience we will identify the group $G$ with its image in $\mathrm{W}(\mathrm{E}_8)$. Also we denote the~image of the group $\Gal\left(\kka / \ka\right)$ in $\mathrm{W}(\mathrm{E}_8)$ by $\Gamma$. The groups $G$ and $\Gamma$ commute.

Note that one can choose eight disjoint $(-1)$-curves on $\XX$ corresponding to \mbox{a blowup $\XX \rightarrow \Pro^2_{\kka}$} in many ways. Therefore the embeddings of $G$ and $\Gamma$ into $\mathrm{W}(\mathrm{E}_8)$ are defined up to conjugacy.

To show that a given del Pezzo surface $X$ is $\ka$-rational or not one should know properties of the group $\Gamma \subset \mathrm{W}(\mathrm{E}_8)$. In this section we study some properties of the group~$\mathrm{W}(\mathrm{E}_8)$ and~its subgroups. In particular, we study conjugacy classes of elements of order $3$ in~$\mathrm{W}(\mathrm{E}_8)$, since the groups of types $\mathrm{II}$, $\mathrm{III}$ and $\mathrm{IX}$ of Theorem \ref{DP1groupclass}, that are listed in Proposition \ref{DP1}, contain elements of order $3$.

We start from description of the set of $(-1)$-curves on a del Pezzo surface of degree $1$.

\begin{theorem}[{cf. \cite[Theorem IV.4.3]{Man74}}]
\label{DP1lines}
Let $\XX$ be a del Pezzo surface of degree $1$, that is a blow up of $\Pro^2_{\kka}$ at eight points $p_1$, $p_2$, $\ldots$, $p_8$ in general position. Then there are exactly $240$ $(-1)$-curves on $\XX$, that are the following:
\begin{itemize}
\item $8$ preimages of the points of the blowup;
\item $28$ proper transforms of the lines passing through two points of the blowup;
\item $56$ proper transforms of the conics passing through five points of the blowup;
\item $56$ proper transforms of the cubics passing through seven points of the blowup, one of which is a double point;
\item $56$ proper transforms of the quartics passing through the eight points of the blowup, three of which are double points;
\item $28$ proper transforms of the quintics passing through the eight points of the blowup, six of which are double points;
\item $8$ proper transforms of the sextics passing through the eight points of the blowup, seven of which are double points and one is a triple point.

\end{itemize}
\end{theorem}

We use the following notation.

\begin{notation}
\label{DP1linesnot}
Let $f: \XX \rightarrow \Pro^2_{\kka}$ be the blowup of $\Pro^2_{\kka}$ at eight points $p_1$, $\ldots$, $p_8$ in general position. Put $E_i = f^{-1}(p_i)$ \mbox{and $L = f^*l$}, where~$l$ is the class of a line on $\Pro^2_{\kka}$. One has
$$
-K_{\XX} \sim 3L - \sum \limits_{i=1}^8 E_i.
$$
In this notation $(-1)$-curves on $\XX$, that are the proper transforms of lines, conics and cubics, have classes
$$
L_{ij} \sim L - E_i - E_j, \qquad Q_{ijk} \sim 2L + E_i + E_j + E_k - \sum \limits_{l = 1}^8 E_l, \qquad C_{i-j} \sim 3L - E_i + E_j - \sum \limits_{k = 1}^8 E_k
$$
respectively. For $(-1)$-curves on $\XX$, that are the proper transforms of quadrics, quintics and sextics, we write $\beta Q_{ijk}$, $\beta L_{ij}$ and $\beta E_i$ respectively.
\end{notation}

Note that there is a (non-normal) subgroup $\SG_8 \subset \mathrm{W}(\mathrm{E}_8)$ that permutes subscripts of~$(-1)$-curves of certain types.

Now we want to study some properties of conjugacy classes of elements of order $3$ in~$\mathrm{W}(\mathrm{E}_8)$. The classification of conjugacy classes in the Weyl group $\mathrm{W}(\mathrm{E}_8)$ was obtained by Frame (see \cite{Fr67}), but for convenience we use the notation of \cite[Table 11]{Car72}.

There are four conjugacy classes of elements of order $3$ in~$\mathrm{W}(\mathrm{E}_8)$ that have Carter graphs~$A_2$, $A_2^2$, $A_2^3$ and $A_2^4$. For an element $g$ with a given Carter graph one can easily compute the invariant Picard number $\rho(\XX)^{\langle g \rangle} = \operatorname{rk}\Pic(\XX)^{\langle g \rangle}$ (see \cite[Subsection 6.1]{DI1}). These numbers equal to $7$, $5$, $3$ and $1$ for the elements of types $A_2$, $A_2^2$, $A_2^3$ and $A_2^4$ respectively.

\begin{lemma}
\label{A2A22}
The elements of types $A_2$ and $A_2^2$ are conjugate to $(123)$ and $(123)(456)$ respectively in~$\SG_8 \subset \mathrm{W}(\mathrm{E}_8)$.
\end{lemma}

\begin{proof}
For an element $g = (123) \in \SG_8$ the group $\Pic(\XX)^{\langle g \rangle}$ is generated by $L$, $E_1 + E_2 + E_3$, $E_4$, $\ldots$, $E_8$, and $\rho(\XX)^{\langle g \rangle} = 7$.

For an element $g = (123)(456) \in \SG_8$ the group $\Pic(\XX)^{\langle g \rangle}$ is generated by $L$, $E_1 + E_2 + E_3$, $E_4 + E_5 + E_6$, $E_7$, $E_8$, and $\rho(\XX)^{\langle g \rangle} = 5$.
\end{proof}

From Section $3$ one can see that three types of elements of order $3$ can act on a del Pezzo surface of degree $1$: elements of types $\mathrm{II}$ and $\mathrm{III}$ (see Table \ref{table1}), and an element~$\alpha$ defined at the beginning of Section $3$. For each of those types we want to find the corresponding Carter graph. For algebraically closed field the classification of elements of order $3$ in~$\mathrm{Cr}_2(\kka)$ is obtained in \cite{dF04}. The following lemma easily follows from the results of this paper.

\begin{lemma}
\label{geom3}
Elements of types $\mathrm{II}$ and $\mathrm{III}$, and an element $\alpha$ have Carter graphs $A_2^3$, $A_2^2$ and $A_2^4$ respectively.
\end{lemma}

\begin{proof}

Let $g$ be an element of order $3$ acting on a del Pezzo surface $\XX$ of degree $1$. Applying $\langle g \rangle$-equivariant MMP one can get a $\langle g \rangle$-morphism $\XX \rightarrow \overline{Y}$, where $\overline{Y}$ is a $\langle g \rangle$-minimal del Pezzo surface.

By \cite[Theorem A]{dF04} there are four possibilities: $\overline{Y} \cong \Pro^2_{\kka}$, $\overline{Y} \cong \Pro^1_{\kka} \times \Pro^1_{\kka}$, $\overline{Y}$ is a cubic surface and $g$ fixes a hyperplane section on $\overline{Y}$, or $\overline{Y} = \XX$ is a del Pezzo surface of degree~$1$ and $g$ acts as $\alpha$. Moreover, degree of $\overline{Y}$ depends on the genus of the curve of $g$-fixed points on $\XX$ (see \cite[Theorem F]{dF04}).

If $g$ has type $\mathrm{II}$ then $g$-fixed curve is an elliptic curve, and $\overline{Y}$ is a cubic surface. Thus $\XX \rightarrow \overline{Y}$ is a contraction of two $(-1)$-curves, $\rho(\XX)^{\langle g \rangle} = 3$, and $g$ has Carter graph $A_2^3$.

If $g$ has type $\mathrm{III}$ then $g$ has only isolated fixed points, and either $\overline{Y} \cong \Pro^2_{\kka}$, or $\overline{Y} \cong \Pro^1_{\kka} \times \Pro^1_{\kka}$. Thus $\XX \rightarrow \overline{Y}$ is a contraction of seven or eight $(-1)$-curves, $\rho(\XX)^{\langle g \rangle} \geqslant 4$, and $g$ has Carter graph $A_2$ or $A_2^2$. If $g$ has Carter graph $A_2$ then it is conjugate to $(123)$ in $\mathrm{W}(\mathrm{E}_8)$ by Lemma~\ref{A2A22}. Therefore $\XX$ can be realised as a blowup of a $\langle g \rangle$-orbit of cardinality $3$ and five $g$-fixed points on $\Pro^2_{\kka}$ in general position. But this is impossible, since an element of~$\mathrm{PGL}_2(\kka)$ cannot have five fixed points in general position. Hence $g$ has Carter graph~$A_2^2$. 

If $g$ is $\alpha$ then $\XX$ is $\langle g \rangle$-minimal, $\rho(\XX)^{\langle g \rangle} = 1$, and $g$ has Carter graph $A_2^4$.

\end{proof}

\begin{corollary}
\label{geom3nonrat}
Let $X$ be a del Pezzo surface of degree $1$. If the group $\Gamma \subset \mathrm{W}(\mathrm{E}_8)$ contains an element of type $A_2^3$ or $A_2^4$ then $X$ is not $\ka$-rational.
\end{corollary}

\begin{proof}
If $\Gamma$ contains an element $g$ of type $A_2^4$ then $\rho(X) = \rho(\XX)^{\Gamma} = \rho(\XX)^{\langle g \rangle} = 1$, and $X$ is minimal. Therefore $X$ is not $\ka$-rational by Theorem \ref{ratcrit}.

If $\Gamma$ contains an element $g$ of type $A_2^3$ then $\Pic(X) = \Pic(\XX)^{\Gamma} \subset \Pic(\XX)^{\langle g \rangle}$. Let $Y$ be a~minimal model of $X$. The action of the element $g$ on $\Pic(\XX)$ is conjugate to the action of an element of type $\mathrm{II}$. Therefore for a del Pezzo surface $\XX'$ with a geometric action of~an element $g'$ of type $\mathrm{II}$ there is a $g'$-equivariant contraction $\XX' \rightarrow \overline{Y}'$, and $K_{\overline{Y}'}^2 = K_Y^2$. The element $g'$ must have an elliptic curve of fixed points on $\overline{Y}'$, since it has an elliptic curve of fixed points on $\XX'$. Thus as in the proof of Lemma \ref{geom3} there is a $g'$-equivariant map $\overline{Y}' \rightarrow \overline{Z}'$, and $\overline{Z}'$ is a cubic surface. One has $K_Y^2 = K_{\overline{Y}'}^2 \leqslant K_{\overline{Z}'}^2 = 3$. Therefore $X \approx Y$ is not $\ka$-rational by Theorem \ref{ratcrit}.

\end{proof}

To apply the result of Corollary \ref{geom3nonrat} we need the following definition.

\begin{definition}
\label{Daviddef}
Six $(-1)$-curves $H_1$, $\ldots$, $H_6$ on a del Pezzo surface of degree $1$ form a~\textit{\davidsstar-configuration} if
$$
H_i \cdot H_{i + 1} = 0, \qquad H_i \cdot H_{i + 2} = 2, \qquad H_i \cdot H_{i + 3} = 3
$$
(Hereinafter in this section all subscripts are considered modulo $6$).

\begin{figure}
\caption{\davidsstar-configuration}
\includegraphics[angle=90]{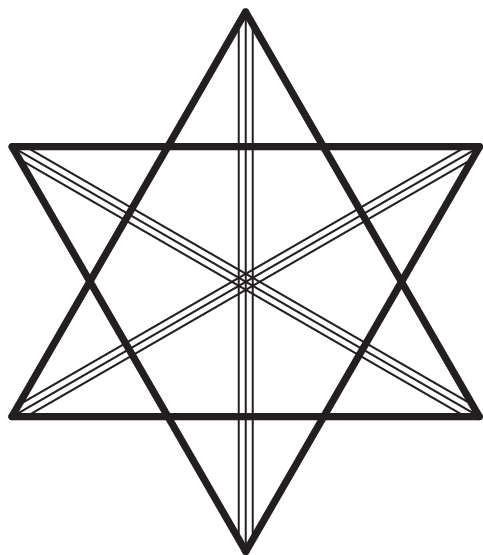}
\end{figure}
\end{definition}

A $\davidsstar$-configuration is depicted in Figure 1: the vertices of a hexagon denote $(-1)$-curves, the bold lines and the triple lines denote intersection with multiplicity $2$ and $3$ respectively.

\begin{example}
\label{Davidexample}
Assume that an element $\alpha$ acts on a del Pezzo surface $X$ of degree $1$. Then any $\langle \alpha, \beta \rangle$-orbit of any $(-1)$-curve is a \davidsstar-configuration.
\end{example}

\begin{remark}
\label{Davidproperties}
Let $(-1)$-curves $H_1$, $\ldots$, $H_6$ form a \davidsstar-configuration. Then
$$
H_i + H_{i+3} \sim -2K_X, \qquad H_i + H_{i+2} + H_{i+4} \sim -3K_X, \qquad \sum \limits_{i = 1}^6 H_i \sim -6K_X.
$$
In particular, $H_{i + 3} = \beta H_i$.

Moreover, for any two disjoint $(-1)$-curves $A$ and $B$ there is a unique \davidsstar-configuration containing $A$ and $B$. This \davidsstar-configuration consists of $(-1)$-curves with classes $A$, $B$, $-K_X - A + B$, $-2K_X -A $, $-2K_X - B$, $-K_X + A - B$.

To show that $H_i + H_{i+3} \sim -2K_X$, one can consider the $(-1)$-curves $\beta H_i$ and $H_{i + 3}$:
$$
\beta H_i \cdot H_{i + 3} = (-2K_X - H_i) \cdot H_{i + 3} = -1,
$$
\noindent therefore the $(-1)$-curve $\beta H_i$ is $H_{i + 3}$.

To show that $H_i + H_{i+2} + H_{i+4} \sim -3K_X$ one can consider the $(-1)$-curve $\beta H_i$:
$$
\beta H_i \cdot H_{i+2} = (-2K_X - H_i) \cdot H_{i + 2} = 0, \qquad \beta H_i \cdot H_{i+4} = (-2K_X - H_i) \cdot H_{i + 4} = 0.
$$
\noindent Therefore we can contract the $(-1)$-curve $\beta H_i$, and get a del Pezzo surface $X_2$ of degree~$2$. The images $\widetilde{H}_{i + 2}$ and $\widetilde{H}_{i + 4}$ of $H_{i+2}$ and $H_{i+4}$ respectively are permuted by the Geiser involution on $X_2$. Therefore $\widetilde{H}_{i + 2} + \widetilde{H}_{i + 4} \sim -K_{X_2}$. Thus one has
$$
H_{i + 2} + H_{i + 4} \sim -K_X + \beta H_i \sim -3K_X - H_i.
$$
\end{remark}

\begin{lemma}
\label{Davidintersection}
Let $A$ be a $(-1)$-curve and $(-1)$-curves $H_1$, $\ldots$, $H_6$ form \mbox{a \davidsstar-configuration}. Then either $A \cdot H_i = 1$ for any $i$, or there exists $k$ such that
$$
A \cdot H_k = A \cdot H_{k + 1} = 0, \qquad A \cdot H_{k + 2} = A \cdot H_{k + 5} = 1, \qquad A \cdot H_{k + 3} = A \cdot H_{k + 4} = 2.
$$
\end{lemma}

\begin{proof}
Note that $A \cdot H_i$ is $0$, $1$ or $2$ for any $i$. One has
$$
A \cdot (H_i + H_{i+2} + H_{i+4}) = -A \cdot 3K_X = 3.
$$
\noindent Therefore either $A \cdot H_i = A \cdot H_{i + 2} = A \cdot H_{i + 4} = 1$, or there exists $k$ such that $A \cdot H_k = 0$, $A \cdot H_{k + 2} = 1$, $A \cdot H_{k + 4} = 2$. The assertion of Lemma \ref{Davidintersection} follows from the following equality
$$
A \cdot H_{i + 3} = A \cdot \beta H_i = A (-2K_X - H_i) = 2 - A \cdot H_i.
$$
\end{proof}

Applying Lemma \ref{Davidintersection} and Remark \ref{Davidproperties} one can check that for two $\davidsstar$-configurations one of the following possibilities holds.

\begin{cordef}
\label{2Daviddef}
Let $(-1)$-curves $A_1$, $\ldots$, $A_6$ and $B_1$, $\ldots$, $B_6$ form \mbox{two \davidsstar-configurations} $\davidsstar_A$ and $\davidsstar_B$. Then up to change of subscripts one of the following possibilities holds.
\begin{enumerate}
\item One has $A_i \cdot B_j = 1$ for any $i$, $j$. In this case we say that $\davidsstar_A$ and $\davidsstar_B$ are \textit{asynchronized} (see Figure 2).

\item One has
$$
A_i \cdot B_i = A_i \cdot B_{i+3} = 1, \qquad A_i \cdot B_{i+1} = A_i \cdot B_{i+2} = 2, \qquad A_i \cdot B_{i-1} = A_i \cdot B_{i-2} = 0.
$$
In this case we say that $\davidsstar_A$ and $\davidsstar_B$ are \textit{synchronized} (see Figure 3).

\item One has
$$
A_2 \cdot B_2 = A_2 \cdot B_3 = A_3 \cdot B_2 = A_3 \cdot B_3 = A_5 \cdot B_5 = A_5 \cdot B_6 = A_6 \cdot B_5 = A_6 \cdot B_6  = 2,
$$
$$
A_2 \cdot B_5 = A_2 \cdot B_6 = A_3 \cdot B_5 = A_3 \cdot B_6 = A_5 \cdot B_2 = A_5 \cdot B_3 = A_6 \cdot B_2 = A_6 \cdot B_3  = 0,
$$
$$
A_1 \cdot B_i = A_4 \cdot B_i = A_i \cdot B_1 = A_i \cdot B_4 = 1,
$$
for any $i$. In this case we say that $\davidsstar_A$ and $\davidsstar_B$ are \textit{abnormal} (see Figure 4).

\end{enumerate}

\begin{figure}
\caption{Asynchronized \davidsstar-configurations}
\includegraphics{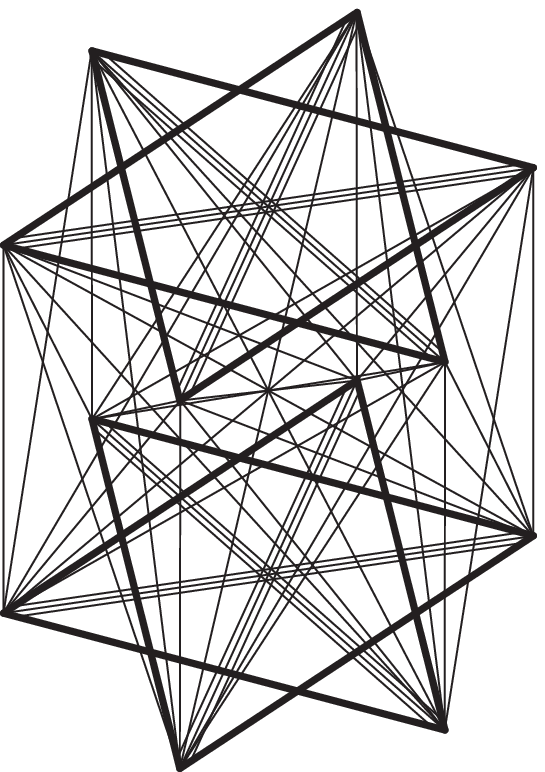}
\end{figure}

\begin{figure}
\caption{Synchronized \davidsstar-configurations}
\includegraphics{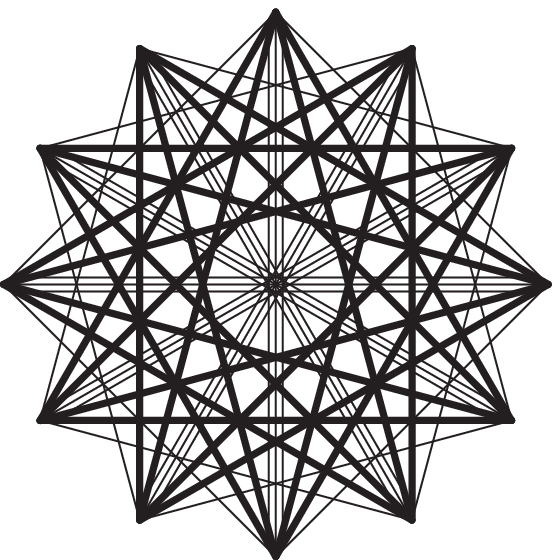}
\end{figure}

\begin{figure}
\caption{Abnormal \davidsstar-configurations}
\includegraphics{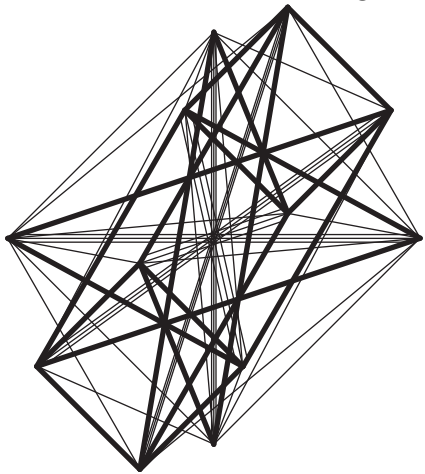}
\end{figure}

\end{cordef}

\begin{remark}
\label{Davidauto}
Note that a $\davidsstar$-configuration has automorphism group $\DG_{12}$ as a graph. A~pair of asynchronized, synchronized or abnormal $\davidsstar$-configurations have automorphism groups $\left(\DG_{12} \times \DG_{12} \right) \rtimes \CG_2$, $\DG_{24}$ and $\DG_8 \times \CG_2$ respectively.

In particular, if an element of order $3$ preserves two $\davidsstar$-configurations and faithfully acts on at least one of them, then these configurations are either asynchronized or synchronized.
\end{remark}

If a \davidsstar-configuration is invariant under the action of an element $g$ of order $3$ then either~$g$ acts trivially on this configuration, or faithfully. Let us compute numbers of $g$-invariant \davidsstar-configurations for different types of $g$.

\begin{lemma}
\label{Davidinv}
Let an element $g$ of order $3$ act on a del Pezzo surface $X$ of degree $1$. Then one has the following.

\begin{itemize}

\item If $g$ has type $A_2$ then there are $72$ invariant $(-1)$-curves, and one $g$-invariant \davidsstar-configuration on which $g$ acts faithfully. All $g$-invariant $(-1)$-curves meet each $(-1)$-curve from this $g$-invariant $\davidsstar$-configuration at a point.

\item If $g$ has type $A_2^2$ then there are $12$ invariant $(-1)$-curves forming two \mbox{\davidsstar-configurations}, and two $g$-invariant \davidsstar-configurations on which $g$ acts faithfully. These four $\davidsstar$-configurations are pairwisely asynchronized.

\item If $g$ has type $A_2^3$ then there are $6$ invariant $(-1)$-curves forming a \mbox{\davidsstar-configuration}, and twelve $g$-invariant \davidsstar-configurations on which $g$ acts faithfully. Each of those twelve $\davidsstar$-configurations is asynchronised with the $\davidsstar$-configuration consisting of the six $g$-invariant $(-1)$-curves.

\item If $g$ has type $A_2^4$ then there are $40$ invariant \davidsstar-configurations on which $g$ acts faithfully.

\end{itemize}

\end{lemma}

\begin{proof}

We use Notation \ref{DP1linesnot}.

By Lemma \ref{A2A22} an element of type $A_2$ is conjugate to $(123) \in \SG_8 \subset \mathrm{W}(\mathrm{E}_8)$. \mbox{A $(-1)$-curve} is $(123)$-invariant if and only if its set of subscripts is $(123)$-invariant. Therefore it is easy to find that there are exactly $72$ invariant $(-1)$-curves on $X$: $E_4$, $\ldots$, $E_8$, $L_{45}$, $\ldots$, $L_{78}$, $Q_{123}$, $Q_{456}$, $\ldots$, $Q_{678}$, $C_{4-5}$, $\ldots$, $C_{8-7}$, $\beta Q_{123}$, $\beta Q_{456}$, $\ldots$, $\beta Q_{678}$, $\beta L_{45}$, $\ldots$, $\beta L_{78}$, $\beta E_4$, $\ldots$, $\beta E_8$.

If a curve $H$ is contained in a $(123)$-invariant \davidsstar-configuration, then by Remark \ref{Davidproperties} one has $H + gH + g^2H \sim -3K_X$. Therefore $H$ is $C_{i-j}$ for certain $i$ and $j$. One can check that the only $(123)$-invariant \davidsstar-configuration, on which $g$ acts faithfully, consists of~$(-1)$-curves $C_{1-2}$, $C_{3-2}$, $C_{3-1}$, $C_{2-1}$, $C_{2-3}$ and $C_{1-3}$.

~

By Lemma \ref{A2A22} an element of type $A_2^2$ is conjugate to $(123)(456) \in \SG_8 \subset \mathrm{W}(\mathrm{E}_8)$. As in the previous case one can check that there are $12$ invariant $(-1)$-curves on $X$ and that these curves form two \davidsstar-configurations: $E_7$, $E_8$, $C_{7-8}$, $\beta E_7$, $\beta E_8$, $C_{8-7}$ and $L_{78}$, $\beta Q_{123}$, $Q_{456}$, $\beta L_{78}$, $Q_{123}$, $\beta Q_{456}$.

Also as in the previous case one can check that there are two $(123)(456)$-invariant \mbox{\davidsstar-configurations}, on which $g$ acts faithfully: $C_{1-2}$, $C_{3-2}$, $C_{3-1}$, $C_{2-1}$, $C_{2-3}$, $C_{1-3}$ and $C_{4-5}$, $C_{6-5}$, $C_{6-4}$, $C_{5-4}$, $C_{5-6}$, $C_{4-6}$. One can check that the constructed four $\davidsstar$-configurations are pairwisely asynchronized.

~

As in the proof of Lemma \ref{geom3} if $g$ has type $A_2^3$ then there exists a $\langle g \rangle$-equivariant contraction $\XX \rightarrow \overline{Y}$ of two $(-1)$-curves, where $\overline{Y}$ is a cubic surface. Therefore \mbox{$\Pic(\XX)^{\langle g \rangle} \cong \mathbb{Z}^3$} is generated by $K_{\XX}$, $E_7$ and $E_8$. Applying Notation \ref{DP1linesnot} one can easily see that the only $g$-invariant $(-1)$-curves are $E_7$, $E_8$, $C_{7-8}$, $\beta E_7$, $\beta E_8$, $C_{8-7}$. These curves form \mbox{a $\davidsstar$-configuration}.

Note that for each pair of $g$-invariant disjoint $(-1)$-curves there exists a $\langle g \rangle$-equivariant contraction $\XX \rightarrow \overline{Y}'$, where $\overline{Y}'$ is $\langle g \rangle$-minimal cubic surface. The $g$-orbit of any line on this surface consists of three meeting each other lines. The preimages of these orbits on~$\XX$ do not have class $-3K_X$. Thus these orbits can not be contained in $g$-equivariant $\davidsstar$-configurations. There are $27$ lines on each cubic surface and $6$ ways to get a cubic surface starting from $\XX$. Moreover, a $(-1)$-curve $A$ on $\XX$ can not be $g$-equivariantly mapped to a line on two such cubic surfaces, since in this case $A$ is disjoint with at least three $g$-invariant $(-1)$-curves on $\XX$, that is impossible by Lemma \ref{Davidintersection}.

Therefore $162$ $(-1)$-curves on $\XX$ are not contained in $\davidsstar$-configurations and six $(-1)$-curves are $g$-invariant. Let us show that the remaining $72$ $(-1)$-curves form twelve $g$-invariant \davidsstar-configurations. By Lemma \ref{Davidintersection} if a $(-1)$-curve $A$ can not be $g$-equivariantly mapped to a line on a cubic surface then $A$ meets each $g$-invariant $(-1)$-curve at a point. Assume that
$$
A + gA + g^2A \sim lK_X + mE_7 + nE_8.
$$
Then one have
$$
-3 = K_X \cdot \left( A + gA + g^2A \right) = K_X \cdot \left( lK_X + mE_7 + nE_8 \right) =  l - m - n,
$$
$$
3 = E_7 \cdot \left( A + gA + g^2A \right) = E_7 \cdot \left( lK_X + mE_7 + nE_8 \right) =  -l - m,
$$
$$
3 = E_8 \cdot \left( A + gA + g^2A \right) = E_8 \cdot \left( lK_X + mE_7 + nE_8 \right) =  -l - n.
$$
\noindent Therefore $l = -3$, $m = n = 0$, and $A + gA + g^2A \sim -3K_X$. One can easily check that
$$
A \cdot gA = gA \cdot g^2A = A \cdot gA^2 = 2.
$$
\noindent Hence the $(-1)$-curves $A$, $gA$, $gA^2$, $\beta A$, $\beta gA$ and $\beta g^2A$ form a $g$-invariant $\davidsstar$-configuration.

~

For type $A_2^4$ the assertion follows from Lemma \ref{geom3} and Example \ref{Davidexample}.

\end{proof}

\begin{corollary}
\label{Davidnotrat}
Let $X$ be a del Pezzo surface of degree $1$. If the group $\Gamma \subset \mathrm{W}(\mathrm{E}_8)$ contains an element $g$ of order $3$ that acts faithfully on three $g$-invariant \davidsstar-configurations then $X$ is not $\ka$-rational.
\end{corollary}

\begin{proof}
By Lemma \ref{Davidinv} the element $g$ has type $A_2^3$ or $A_2^4$, and by Corollary \ref{geom3nonrat} in these cases $X$ is not $\ka$-rational.
\end{proof}

We give one more lemma that is useful to prove that a del Pezzo surface of degree $1$ is not $\ka$-rational. Actually, we do not use this lemma further, but put it here to show that the notion of $\davidsstar$-configuration is useful not only for elements of order $3$.

\begin{lemma}
\label{Davidnotrateven}
Let $X$ be a del Pezzo surface of degree $1$. If the group $\Gamma \subset \mathrm{W}(\mathrm{E}_8)$ contains an element $g$ that acts faithfully on a $g$-invariant \davidsstar-configuration and for any curve $H$ in this configuration one has $H \cdot gH = 3$ then $X$ is not $\ka$-rational.
\end{lemma}

\begin{proof}
Obviously $g$ has even order. The conditions of Lemma \ref{Davidnotrateven} hold for odd powers of~$g$, therefore without loss of generality we may assume that order of $g$ is $2^k$.

Assume that $X$ is $\ka$-rational and consider $\langle g \rangle$-MMP for $\XX$. By Theorem \ref{ratcrit} there exists a $\Gamma$-equivariant map $\XX \rightarrow \overline{Y}'$ to a $\Gamma$-minimal del Pezzo surface $\overline{Y}'$ with $K_{\overline{Y}'}^2 \geqslant 5$. Therefore there exists a $\langle g \rangle$-invariant map $\XX \rightarrow \overline{Y}$ to a $\langle g \rangle$-minimal del Pezzo surface $\overline{Y}$ with $K_{\overline{Y}}^2 \geqslant 5$.

It is well known that a del Pezzo surface of degree $5$ or $6$ cannot be $\langle g \rangle$-minimal if $g$ has order $2^k$. Therefore either $\overline{Y} \cong \Pro^1_{\kka} \times \Pro^1_{\kka}$, or $\overline{Y} \cong \Pro^2_{\kka}$. In the former case $\XX \rightarrow \overline{Y}$ is \mbox{a $\langle g \rangle$-equivariant} contraction of seven $(-1)$-curves. Thus one of this curves \mbox{is $\langle g \rangle$-invariant}, and one can $\langle g \rangle$-equivariant contract the other six $(-1)$-curves and get a~del Pezzo surface $\overline{Z}$ of degree $7$. There is a $\langle g \rangle$-equivariant morphism $\overline{Z} \rightarrow \Pro^2_{\kka}$. So this case can be reduced to the latter case $\overline{Y} \cong \Pro^2_{\kka}$.


If $\overline{Y} \cong \Pro^2_{\kka}$ then $\XX \rightarrow \overline{Y}$ is a $\langle g \rangle$-equivariant contraction of eight $(-1)$-curves. Therefore~$g$ is an element of order $2^k$ in $\SG_8 \subset \mathrm{W}(\mathrm{E}_8)$. In Notation \ref{DP1linesnot} any pair of $(-1)$-curves $H$ and~$gH$ such that $H \cdot gH = 3$, is a pair $C_{i-j}$ and $C_{j-i}$ for certain subscripts $i$ and $j$. Moreover, these pairs do not have common subscripts. One can easily see that such pairs cannot form a \davidsstar-configuration. Thus we have a contradiction, and $X$ is not $\ka$-rational.

\end{proof}

Now for elements of order $3$ we show how to proof $\ka$-rationality and $G$-minimality in terms of $\davidsstar$-configurations.

\begin{lemma}
\label{Davidmin}
Let $\XX$ be a del Pezzo surface of degree $1$, and $G$ be a group acting on $\XX$, such that there are four $G$-invariant pairwisely asynchronized $\davidsstar$-configurations $\davidsstar_1$, $\davidsstar_2$, $\davidsstar_3$ and $\davidsstar_4$. If for each $\davidsstar_i$ there is an element $g_i \in G$ of order $3$ faithfully acting on $\davidsstar_i$ then~$\rho(\XX)^G = 1$.
\end{lemma}

\begin{proof}
Let $A_i$ and $B_i$ be two disjoint $(-1)$-curves from $\davidsstar_i$. Let us show that $8$-dimensional vector space $V = \Pic(\XX) \otimes \mathbb{Q} \cap K_{\XX}^{\perp}$ is generated by $A_i + K_{\XX}$ and $B_i + K_{\XX}$. Obviously,
$$
\left( A_i + K_{\XX} \right)^2 = \left( B_i + K_{\XX} \right) = -2, \qquad \left( A_i + K_{\XX} \right)\left( B_i + K_{\XX} \right) = -1.
$$
Moreover, for $i \ne j$ one has
$$
\left( A_i + K_{\XX} \right)\left( A_j + K_{\XX} \right) = \left( A_i + K_{\XX} \right)\left( B_j + K_{\XX} \right) = \left( B_i + K_{\XX} \right)\left( B_j + K_{\XX} \right) = 0.
$$

Assume that
$$
D = a_1 \left( A_1 + K_{\XX} \right) + b_1 \left( B_1 + K_{\XX} \right) + \ldots = 0.
$$
Then
$$
D \cdot \left( A_i + K_{\XX} \right) = -2a_i - b_i = 0, \qquad D \cdot \left( B_i + K_{\XX} \right) = -a_i - 2b_i = 0.
$$
Therefore for any $i$ one has $a_i = b_i = 0$. It means that $A_i + K_{\XX}$ and $B_i + K_{\XX}$ are linearly independent, and generate $V$.

The group $G$ faithfully acts on the $\davidsstar$-configurations $\davidsstar_1$, $\davidsstar_2$, $\davidsstar_3$ and $\davidsstar_4$. Thus $V$ can be decomposed into direct sum of four two-dimensional $G$-invariant vector spaces \mbox{$V_i = \langle A_i + K_{\XX}, B_i + K_{\XX} \rangle$}. The group $\langle g_i \rangle$ faithfully acts on $V_i$, but a group of order $3$ does not have one-dimensional representations over $\mathbb{Q}$. Thus $V_i^G = V_i^{\langle g_i \rangle} = 0$ for any $i$. Therefore $V^G = 0$, and $\XX$ is $G$-minimal.

\end{proof}

Applying Lemma \ref{Davidinv} we immediately have the following two corollaries.

\begin{corollary}
\label{Davidmin1}
Let $\XX$ be a del Pezzo surface of degree $1$, and $g$ and $h$ be two commuting elements of order $3$ in $\mathrm{W}(\mathrm{E}_8)$. If $g$ has type $A_2^3$, and $h$ maps a $g$-invariant $(-1)$-curve to the other curve, then $\XX$ is $\langle g, h \rangle$-minimal.
\end{corollary}


\begin{remark}
One can easily see that the conditions of Corollary \ref{Davidmin1} are equivalent to the fact that the element $h$ faithfully acts on the $\davidsstar$-configuration, consisting of $g$-invariant $(-1)$-curves.
\end{remark}

\begin{corollary}
\label{Davidmin2}
Let $\XX$ be a del Pezzo surface of degree $1$, and $g$ and $h$ be two commuting elements of order $3$ in $\mathrm{W}(\mathrm{E}_8)$. If $g$ has type $A_2^2$, and $h$ faithfully acts on the both \mbox{$\davidsstar$-configurations}, consisting of $g$-invariant $(-1)$-curves, then $\XX$ is $\langle g, h \rangle$-minimal.
\end{corollary}


The following lemma is useful to prove $\ka$-rationality of del Pezzo surfaces.

\begin{lemma}
\label{RatLemma}
Let $X$ be a del Pezzo surface and $A$, $B$ and $C$ be a triple of $(-1)$-curves on $X$ defined over $\ka$, such that $A \cdot B = B \cdot C = 1$, and $A \cdot C = 0$. Then $X$ is $\ka$-rational.
\end{lemma}

\begin{proof}
Let us consider a complete linear system $\mathcal{L} = |A + B + C|$. For a general member $D$ of $\mathcal{L}$ one has $D^2 = 1$ and $D \cdot K_X = -3$. By the Riemann--Roch theorem one has
$$
\operatorname{dim} H^0 (X , D) = \frac{1}{2}D \cdot(D - K_X) + 1 = 3.
$$
Therefore $\mathcal{L}$ gives a birational map $X \dashrightarrow \Pro^2_{\ka}$.

\end{proof}

\begin{corollary}
\label{RatCor}
Let $X$ be a del Pezzo surface of degree $1$. If all $(-1)$-curves in two asynchronized $\davidsstar$-configurations are defined over $\ka$ then $X$ is $\ka$-rational.
\end{corollary}

\section{Examples}

In this section we construct explicit examples of quotients of del Pezzo surfaces $X$ of degree $1$ by finite groups $G$ listed in Proposition \ref{DP1} such that $\rho(X)^G = 1$. If $G$ is trivial then $X$ is non-$\ka$-rational by Theorem~\ref{ratcrit}, and if $G$ is of type $\mathrm{I}$ then $X$ is non-$\ka$-rational by \cite[Remark 5.2]{Tr18b} and $X / G$ is non-$\ka$-rational by Remark \ref{DP1typeImin}. For groups of types~$\mathrm{II}$, $\mathrm{III}$ and $\mathrm{IX}$ we show that each of the four possibilities of $\ka$-rationality of $X$ and $X / G$ is realized for certain $\ka$: the surface $X$ can be $\ka$-rational and $X / G$ can be non-$\ka$-rational, $X$ can be non-$\ka$-rational and $X / G$ can be non-$\ka$-rational, $X$ can be $\ka$-rational and $X / G$ can be $\ka$-rational, $X$ can be non-$\ka$-rational and $X / G$ can be $\ka$-rational.

Assume that the field $\ka$ contains $\omega$. Let $X$ be a del Pezzo surface of degree $1$ given in~$\Pro_{\ka}(1 : 1 : 2 : 3)$ by the equation
\begin{equation}
\label{exequation}
Ax^6 + 2Bx^3y^3 +Cy^6 +2z^3 - t^2 = 0.
\end{equation}

A finite group $\CG_3^2$ generated by
$$
g_y: (x : y : z : t) \mapsto (\omega x : y : z : t) \qquad \textrm{and} \qquad g_x: (x : y : z : t) \mapsto (x : \omega y : z : t)
$$
\noindent acts on $X$. The elements $g_{\mathrm{III}} = g_x^2g_y$ and $\alpha = g_xg_y$ lie in $\CG_3^2$.

The set of $g_{\mathrm{III}}$-fixed points consists of five isolated fixed points: $p = (0: 0 : 2 : 4)$ and
$$
q_{x1} = \left(0 : 1 : 0 : \sqrt{C}\right), \quad q_{x2} = \left(0 : 1 : 0 : -\sqrt{C}\right),
$$
$$
q_{y1} = \left(1 : 0 : 0 : \sqrt{A}\right), \quad q_{y2} = \left(1 : 0 : 0 : -\sqrt{A}\right).
$$
\noindent A linear system $\mathcal{L}_{\mathrm{III}}$ given by $\lambda xy = \mu z$ consists of $g_{\mathrm{III}}$-invariant curves.

The set of $g_y$-fixed points consists of the curve $C_x$ given by $x = 0$, and two isolated fixed points $q_{y1}$ and $q_{y2}$. A linear system $\mathcal{L}_{y}$ given by $\lambda y^2 = \mu z$ consists of $g_y$-invariant curves.

The set of $g_x$-fixed points consists of the curve $C_y$ given by $y = 0$, and two isolated fixed points $q_{x1}$ and $q_{x2}$. A linear system $\mathcal{L}_{x}$ given by $\lambda x^2 = \mu z$ consists of $g_x$-invariant curves.

Let us find reducible members in the linear systems $\mathcal{L}_x$, $\mathcal{L}_y$ and $\mathcal{L}_{\mathrm{III}}$. For simplicity in all cases we put $\mu = 1$. We start from the linear system $\mathcal{L}_x$:
$$
(A + 2\lambda^3)x^6 + 2Bx^3y^3 + Cy^6 - t^2 = 0.
$$
This section is reducible if $(A + 2\lambda^3)x^6 + 2Bx^3y^3 + Cy^6$ is a square in $\ka(x , y)$. Therefore one has
$$
B^2 = (A + 2\lambda^3)C,
$$
and components of reducible fibres are given by
$$
\lambda^3 = \frac{B^2 - AC}{2C}, \qquad z = \lambda x^2, \qquad t = \pm \sqrt{C} \left( \frac{B}{C} x^3 + y^3 \right).
$$
Note that these components are $(-1)$-curves that are transitively permuted by the group~$\langle \alpha, \beta \rangle$. Therefore these curves form $\davidsstar$-configuration (see Example \ref{Davidexample}). We denote this configuration by $\davidsstar_x$.

Similarly, components of reducible fibres of $\mathcal{L}_y$ are given by
$$
\lambda^3 = \frac{B^2 - AC}{2A}, \qquad z = \lambda y^2, \qquad t = \pm \sqrt{A} \left( x^3 + \frac{B}{A} y^3 \right).
$$
These components form $\davidsstar$-configuration. We denote this configuration by $\davidsstar_y$.

Now consider reducible members of the linear system $\mathcal{L}_{\mathrm{III}}$:
$$
Ax^6 + 2(B + \lambda^3)x^3y^3 + Cy^6 - t^2 = 0.
$$
This section is reducible if $Ax^6 + 2(B + \lambda^3)x^3y^3 + Cy^6$ is a square in $\ka(x , y)$. Therefore one has
$$
(B + \lambda^3)^2 = AC,
$$
and components of reducible fibres are given by
$$
\lambda^3 = -B \pm \sqrt{AC}, \qquad z = \lambda xy, \qquad t = \pm \sqrt{A} x^3 \pm \sqrt{C} y^3.
$$
These $12$ components are $(-1)$-curves that are contained in two $\langle \alpha, \beta \rangle$-orbits. Therefore these curves form two $\davidsstar$-configurations. We denote these configurations by $\davidsstar_{\mathrm{III}+}$ and~$\davidsstar_{\mathrm{III}-}$. Moreover, $g_{\mathrm{III}}$ faithfully acts on $\davidsstar_x$ and $\davidsstar_y$. Therefore by Lemma \ref{Davidinv} the $\davidsstar$-configurations $\davidsstar_x$, $\davidsstar_y$, $\davidsstar_{\mathrm{III}+}$ and~$\davidsstar_{\mathrm{III}-}$ are pairwisely asynchronized.

The following example gives all possibilities for the quotients of non-$\ka$-rational del Pezzo surfaces of degree $1$.

\begin{example}
\label{nratto}
Assume that a del Pezzo surface $X$ of degree $1$ is given by equation~\eqref{exequation}, and $B = 0$, $A = C$, such that $\sqrt[3]{A} \notin \ka$ and $\sqrt[3]{\dfrac{A}{2}} \notin \ka$. Then the group $\Gal\left(\ka\left(\sqrt[3]{A}, \sqrt[3]{2} \right) / \ka \right)$ contains an element $h$ of order $3$ that faithfully acts on \mbox{the $\davidsstar$-configurations} $\davidsstar_x$, $\davidsstar_y$, $\davidsstar_{\mathrm{III}+}$ and $\davidsstar_{\mathrm{III}-}$. Therefore $X$ is not $\ka$-rational by Corollary~\ref{Davidnotrat} and $\rho(X) = 1$ by Lemma \ref{Davidmin}.

The element $g_y: (x : y : z : t) \mapsto (\omega x : y : z : t)$ acts on $X$ and has type $\mathrm{II}$. The set of isolated fixed points of $g_y$ consists of $q_{y1}$ and $q_{y2}$. Therefore by Remark \ref{DP1typeIImin} the quotient~$X / \langle g_y \rangle$ is $\ka$-rational if $A$ is a square in $\ka$, and non-$\ka$-rational otherwise.

The element $g_{\mathrm{III}}: (x : y : z : t) \mapsto (\omega x : \omega^2 y : z : t)$ acts on $X$ and has type~$\mathrm{III}$. The set of fixed points of $g_{\mathrm{III}}$ consists of $p$, $q_{x1}$, $q_{x2}$, $q_{y1}$ and $q_{y2}$. Therefore by Remark~\ref{DP1typeIIImin} the quotient $X / \langle g_{\mathrm{III}} \rangle$ is $\ka$-rational if $A$ is a square in $\ka$, and non-$\ka$-rational otherwise.

The element $g_{\mathrm{IX}}: (x : y : z : t) \mapsto (-\omega x : \omega^2 y : z : t)$ acts on $X$ and has type $\mathrm{IX}$. By Remark \ref{DP1typeIXmin} the quotient $X / \langle g_{\mathrm{IX}} \rangle$ is $\ka$-rational if $A$ is a square in $\ka$, and non-$\ka$-rational otherwise.
\end{example}

\begin{remark}
\label{nrattoremark}
Note that if there is at least one non-square $u$ in $\ka$, and at least one non-cube~$v$ in $\ka$, then one can find $A$ satisfying conditions of any case of Example \ref{nratto}. We want to find $A$, such that $\sqrt[3]{A} \notin \ka$, $\sqrt[3]{\dfrac{A}{2}} \notin \ka$ and $A$ is either square or not.

In this case one can find an element $w \in \ka$, that is not a square, and is not a cube in~$\ka$. If $v$ is not a square then one can put $w = v$, and otherwise $w = vu^3$.

Therefore if $\sqrt[3]{2} \in \ka$ then $A = w^2$ and $A = w$ give the required possibilities.

If $\sqrt[3]{2} \notin \ka$ and $\sqrt{2} \in \ka$ then one can find $k \in \{1, 2, 3\}$, such that $2^{k-1}w$ and $2^k w$ are~not cubes in $\ka$. For this $k$ the numbers $2^{2k} w^2$ and $2^{2k + 1} w^2$ are not cubes too. \mbox{Therefore $A = 2^{2k + 1} w^2$} and $A = 2^k w$ give the required possibilities.

If $\sqrt[3]{2} \notin \ka$ and $\sqrt{2} \notin \ka$ then one can put $A = 4$ and $A = 32$ to achieve the required possibilities. 

\end{remark}

Almost the same construction gives an example of a $\ka$-rational quotient of a $\ka$-rational del Pezzo surface of degree $1$ by a group of type $\mathrm{II}$.

\begin{example}
\label{rattorat2}

Assume that a del Pezzo surface $X$ of degree $1$ is given by equation~\eqref{exequation}, and $B = 0$, $C = A^{-1}$, such that $\sqrt{A} \in \ka$ and $\sqrt[3]{2A} \notin \ka$. Then all $(-1)$-curves in $\davidsstar_{\mathrm{III}+}$ and~$\davidsstar_{\mathrm{III}-}$ are defined over $\ka$, and $X$ is $\ka$-rational by Corollary \ref{RatCor}.

Moreover, the Galois group $\Gal\left(\ka\left(\sqrt[3]{2A} \right) / \ka \right)$ faithfully acts on the $\davidsstar$-configuration~$\davidsstar_y$ consisting of the six $g_y$-invariant $(-1)$-curves. Thus $\rho(X)^{\langle g_y \rangle} = 1$ by Corollary \ref{Davidmin1}. The quotient $X / \langle g_y \rangle$ is $\ka$-rational by Remark \ref{DP1typeIImin}.

\end{example}

\begin{remark}
\label{rattorat2remark}

Note that if there is at least one non-cube $v$ in $\ka$, then one can easily find $A$ satisfying conditions of Example \ref{rattorat2}. If $\sqrt[3]{2} \notin \ka$ then one can put $A = 1$, and otherwise one can put $A = v^2$.

\end{remark}

To construct the examples of quotients of $\ka$-rational del Pezzo surfaces of degree $1$ by groups of type $\mathrm{III}$ and $\mathrm{IX}$ we have to modify equation~\eqref{exequation}.

\begin{example}
\label{ratto39}

Let $X$ be a del Pezzo surface $X$ of degree $1$ given by the equation
$$
Ax^6 + Cy^6 + 6Fx^2y^2z +2z^3 - t^2 = 0.
$$
Then the element $g_{\mathrm{III}}: (x : y : z : t) \mapsto (\omega x : \omega^2 y : z : t)$ acts on $X$ and has type $\mathrm{III}$. The twelve $g_{\mathrm{III}}$-invariant lines are given by
$$
\lambda^3 + 3F\lambda \pm \sqrt{AC} = 0, \qquad z = \lambda xy, \qquad t = \pm \sqrt{A} x^3 \pm \sqrt{C} y^3.
$$
\noindent These lines form two $\davidsstar$-configurations $\davidsstar_{\mathrm{III}+}$ and $\davidsstar_{\mathrm{III}-}$.

Also one can find that two $\davidsstar$-configurations $\davidsstar_x$ and $\davidsstar_y$, on which $g_{\mathrm{III}}$ acts faithfully, are given by
$$
\lambda^3 = - \frac{A}{2}, \qquad z = \lambda x^2 - \frac{F}{\lambda} y^2, \qquad t = \pm \sqrt{\frac{AC + 4F^3}{A}} \cdot y^3,
$$
\noindent and
$$
\lambda^3 = - \frac{C}{2}, \qquad z = - \frac{F}{\lambda} x^2 + \lambda y^2, \qquad t = \pm \sqrt{\frac{AC + 4F^3}{C}} \cdot x^3.
$$

Assume that $\sqrt{AC} \in \ka$, the equation $\lambda^3 + 3F\lambda + \sqrt{AC} = 0$ has no root defined over~$\ka$, $\sqrt{\dfrac{AC + 4F^3}{A}} \in \ka$, $\sqrt[3]{\dfrac{A}{2}} \in \ka$ and $\sqrt[3]{\dfrac{C}{2}} \in \ka$. Then all $(-1)$-curves in $\davidsstar_x$ and $\davidsstar_y$ are defined over $\ka$, and $X$ is $\ka$-rational by Corollary \ref{RatCor}. Moreover, the Galois group of the equation $\lambda^3 + 3F\lambda + \sqrt{AC} = 0$ contains an element $h$ of order $3$ that faithfully acts on \mbox{the $\davidsstar$-configurations} $\davidsstar_{\mathrm{III}+}$ and $\davidsstar_{\mathrm{III}-}$ consisting of the twelve $g_{\mathrm{III}}$-invariant $(-1)$-curves. Thus $\rho(X)^{\langle g_{\mathrm{III}} \rangle} = 1$ by Corollary \ref{Davidmin2}.

The set of $g_{\mathrm{III}}$-fixed points consists of five isolated fixed points: $p = (0: 0 : 2 : 4)$ and
$$
q_{x1} = \left(0 : 1 : 0 : \sqrt{C}\right), \quad q_{x2} = \left(0 : 1 : 0 : -\sqrt{C}\right),
$$
$$
q_{y1} = \left(1 : 0 : 0 : \sqrt{A}\right), \quad q_{y2} = \left(1 : 0 : 0 : -\sqrt{A}\right).
$$
\noindent Therefore by Remark \ref{DP1typeIIImin} the quotient $X / \langle g_{\mathrm{III}} \rangle$ is $\ka$-rational if $A$ is a square in $\ka$, and not $\ka$-rational otherwise.

~

The element $g_{\mathrm{IX}}: (x : y : z : t) \mapsto (-\omega x : \omega^2 y : z : t)$ acts on $X$ and has type $\mathrm{IX}$. As in the previous case $\rho(X)^{\langle g_{\mathrm{IX}} \rangle} = 1$. By Remark \ref{DP1typeIXmin} the quotient $X / \langle g_{\mathrm{IX}} \rangle$ is $\ka$-rational if $A$ is a square in $\ka$, and not $\ka$-rational otherwise.

\end{example}

\begin{remark}
\label{ratto39remark}
Note that the conditions of Example \ref{ratto39} hold for $A = 2u^3$, $C = 2u^{-3}$, $\ka = \mathrm{k}\left(F, \sqrt{2u(1+F^3)}\right)$, where $\omega \in \mathrm{k}$, and $2u$ is either square for the case of $\ka$-rational quotient, or not square for the case of non-$\ka$-rational quotient.

\end{remark}

The most difficult case is a non-$\ka$-rational quotient of a $\ka$-rational del Pezzo surface of degree $1$ by a group of type $\mathrm{II}$.

\begin{example}
\label{rattonrat2}

Let $X$ be a del Pezzo surface $X$ of degree $1$ given by the equation
$$
Ax^6 + Cy^6 + Ey^4z + z^3 - t^2 = 0.
$$
Then the element $g_y: (x : y : z : t) \mapsto (\omega x : y : z : t)$ acts on $X$ and has type $\mathrm{II}$. The six $g_y$-invariant lines are given by
$$
\lambda^3 + E\lambda + C = 0, \qquad z = \lambda y^2, \qquad t = \pm \sqrt{A} x^3.
$$
\noindent These lines form a $\davidsstar$-configuration $\davidsstar_y$.

Also one can find twelve $\davidsstar$-configurations on which $g_y$ acts faithfully. Four of these $\davidsstar$-configurations consist of $(-1)$-curves invariant under the involution $x \mapsto -x$. These curves are given by
$$
\lambda^3 = A, \qquad z + \lambda x^2 - \frac{w^2}{3} y^2 = 0, \qquad t = \frac{3E - w^4}{6w} y^3 + wyz,
$$
where $w$ is a root of the equation
\begin{equation}
\label{Wequat}
w^8 + 18Ew^4 + 108Cw^2 - 27E^2 = 0.
\end{equation}
The other eight $\davidsstar$-configurations consist of $(-1)$-curves given by
$$
u^2 = \frac{A(27E + 9w^4)}{27E + w^4}, \qquad \lambda^3 = - \frac{8Aw^4}{27E+w^4},
$$
$$
z + \lambda x^2 + \frac{2uw}{3\lambda} xy - \frac{w^2}{3} y^2 = 0, \qquad t = u x^3 + \frac{3E - w^4}{6w} y^3 + wyz, 
$$
where $w$ is a root of equation \eqref{Wequat}.

One can check that for a given value of $w^2$ from equation \eqref{Wequat} three corresponding $\davidsstar$-configurations are pairwisely asynchronized. Moreover, all these configurations are asynchronized with $\davidsstar_y$ by Lemma \ref{Davidinv}.

Assume that $C = \dfrac{27E^2 -18E - 1}{108}$ and $A = (3E + 1)^3(27E + 1)$. Then equation~\eqref{Wequat} has roots $w = 1$ and $w = -1$. Moreover, for these roots $u^2 = 9(3E + 1)^4$ and $\lambda^3 = - 8(3E + 1)^3$. It means that there are two asynchronized $\davidsstar$-configurations, such that all $(-1)$-curves in these configurations are defined over $\ka$. Thus $X$ is $\ka$-rational by Corollary \ref{RatCor}.

The Galois group of the equation $\lambda^3 + E\lambda + C = 0$ depends on presence \mbox{of $\sqrt{\Delta} = \sqrt{729C^2 + 108E^3}$} and $\sqrt[3]{\Xi} = \sqrt[3]{\frac{27C \pm \sqrt{\Delta}}{2}}$ in $\ka$. We have
$$
\Delta = \frac{(3E + 1)^3(27E + 1)}{16}, \textrm{and}\,\,\Xi = \frac{1}{8}\left( 27E^2 -18E - 1 \pm (3E + 1)\sqrt{(3E + 1)(27E + 1)}\right).
$$
Therefore if $\Xi$ is not a cube in $\ka(\sqrt{\Delta})$ then the Galois group of the equation \mbox{$\lambda^3 + E\lambda + C = 0$} contains an element $h$ of order $3$ that faithfully acts on the \mbox{$\davidsstar$-configuration} $\davidsstar_y$. Thus~$\rho(X)^{\langle g_y \rangle} = 1$ by Corollary \ref{Davidmin1}.

Moreover, the set of isolated fixed points of $g_y$ consists of two points
$$
q_{y1} = (1 : 0 : 0 : \sqrt{A}), \quad q_{y2} = (1 : 0 : 0 : -\sqrt{A}).
$$
\noindent Therefore by Remark \ref{DP1typeIImin} the quotient $X / \langle g_y \rangle$ is not $\ka$-rational if $A = (3E + 1)^2 \cdot \Delta$ is not a square in $\ka$.

These conditions hold, for example, for any field $\ka = \mathrm{k}(E)$, where $\omega \in \mathrm{k}$.

\end{example}

\section{Proofs of some main results}

In this section we prove some results which follow from classification of non-$\ka$-rational quotients of del Pezzo surfaces. At first we recall some facts about non-$\ka$-rational nontrivial quotients of cubic surfaces.

\begin{lemma}[{\cite[cf. Lemma 3.2]{Tr16b}}]
\label{DP3C35}
Let a finite group $G$ act on a del Pezzo surface~$X$ of degree $3$ and $N \cong \CG_3$ be a normal subgroup in $G$ such that $N$ has only isolated fixed points. Then the surface $X / N$ is $G / N$-birationally equivalent to a cubic surface of type~$\mathrm{VIII}$. Moreover, the set of $\ka$-points on $X / N$ is dense.
\end{lemma}

\begin{proof}
If a group $N$ of order $3$ acts on a cubic surface $X$ and has only ioslated fixed points, then one can choose coordinates in $\Pro^3_{\kka}$ such that $\XX$ is given by
$$
x^3 + y^3 + zt(ux + vy) + z^3 + t^3 = 0
$$
\noindent and $N$ acts as
$$
(x : y : z: t) \mapsto \left( x : y : \omega z : \omega^2 t \right).
$$
\noindent The fixed points of $N$ lie on the line $z = t = 0$. Thus $N$ has three fixed points $q_1$, $q_2$ and $q_3$. One can easily check that on the tangent spaces of $\XX$ at these points $N$ acts as~$\langle\operatorname{diag}(\omega, \omega^2)\rangle$. Denote by $C_1$ and $C_2$ invariant curves $z = 0$ and $t = 0$ each passing through the three points $q_i$.

Let $f: X \rightarrow X / N$ be the quotient morphism and
$$
\pi: \widetilde{X / N} \rightarrow X / N
$$
\noindent be the minimal resolution of singularities. The curves $f(C_1)$ and $f(C_2)$ meet each other at the three singular points of $X / N$ and $f(C_1) \cdot f(C_2) = 1$. Thus two curves $\pi^{-1}_* f (C_j)$ are disjoint. Moreover by Remark \ref{singularities}, one has
$$
\pi^{-1}_* f (C_j)^2 = f(C_j)^2 - 3 \cdot \frac{2}{3} = \frac{1}{3} C_j^2 - 2 = -1.
$$
Therefore we can $G / N$-equivariantly contract the two $(-1)$-curves $\pi^{-1}_* f (C_j)$ and get a surface~$Y$ with $K_Y^2 = 3$. We denote this contraction by $h: \widetilde{X / N} \rightarrow Y$.

The surface $X / N$ has only Du Val singularities. Therefore $X / N$ is a singular del Pezzo surface and $\widetilde{X / N}$ is a weak del Pezzo surface containing exactly six curves $\pi^{-1}(q_i)$ whose selfintersection is less than $-1$. Thus $Y$ does not contain curves with selfintersection less than $-1$. So $Y$ is a del Pezzo surface of degree $3$.

Denote the six irreducible components of $\pi^{-1}(q_i)$ by $T_{ij}$. One has
$$
T_{ij}^2 = -2, \qquad T_{i1} \cdot T_{i2} = 1, \qquad T_{ij} \cdot \pi^{-1}_* f (C_j) = 1.
$$
\noindent The six curves $hT_{ij}$ are lines on $Y$. Three lines $hT_{i1}$ pass through $h\pi^{-1}_* f (C_1)$, and three lines $hT_{i2}$ pass through $h\pi^{-1}_* f (C_2)$. Therefore $h\pi^{-1}_* f (C_j)$ are Eckardt points on $Y$.

The line passing through $h\pi^{-1}_* f (C_1)$ and $h\pi^{-1}_* f (C_2)$ meet $Y$ at the third point $R$, defined over $\ka$. Moreover, a plane spanned on $hT_{i1}$ and $hT_{i2}$ passes through $R$. This plane cuts three lines from $Y$ with classes $hT_{i1}$, $hT_{i2}$ and $-K_Y - hT_{i1} - hT_{i2}$. Therefore three lines with classes $-K_Y - hT_{i1} - hT_{i2}$ pass through~$R$, and $R$ is an Eckardt point on $Y$. Thus $Y$ has type $\mathrm{VIII}$.

The point $R$ is defined over $\ka$, hence $Y$ is $\ka$-unirational by Remark \ref{unirationality}, and the sets of $\ka$-points on $Y$ and $X / N$ are dense.

\end{proof}

\begin{remark}
\label{DP3C35remark}
Note that the Eckardt points $h\pi^{-1}_* f (C_1)$ and $h\pi^{-1}_* f (C_2)$ are defined over~$\ka$ if and only if the $N$-invariant curves $C_1$ and $C_2$ are defined over $\ka$. One can choose coordinates in $\Pro^3_{\ka}$ such that $N$ acts as
$$
(x : y : z: t) \mapsto \left( x : y : t : -z - t \right),
$$
and the curves $C_1$ and $C_2$ are given by $z = \omega t$ and $z = \omega^2 t$. Therefore these curves and the Eckardt points $h\pi^{-1}_* f (C_j)$ are defined over $\ka$ if and only if $\omega \in \ka$.
\end{remark}

\begin{remark}
\label{DP3C35min}
Note that if $\rho(X)^G = 1$ and the three isolated fixed points of $N$ are permuted by the \mbox{group $G \times \Gal \left( \kka / \ka \right)$} then $Y$ is $(G / N)$-minimal, since
$$
\rho(Y)^{G / N} = \rho\left(\widetilde{X / N}\right)^{G / N} + k = \rho\left(X / N\right)^{G / N} = \rho(X)^G = 1,
$$
\noindent where $k = 2$ if $C_1$ and $C_2$ are defined over $\ka$ and $k = 1$ otherwise. 
\end{remark}

Note that in \cite[Lemma 5.11 and Example 6.4]{Tr16b} there is constructed an example of non-$\ka$-rational quotient of $\CG_3$-minimal $\ka$-rational cubic surface by a group $\CG_3$. But this quotient is a minimal del Pezzo surface of degree $4$ admitting a structure of a conic bundle. Now we show that the constructed in Lemma \ref{DP3C35} cubic surface $Y$ can be minimal.

\begin{example}
\label{DP3C35example}
Let the field $\ka$ contains $\omega$ and $\alpha$, such that $\sqrt[3]{\alpha} \notin \ka$. Consider a cubic surface $X$ given by
$$
\alpha x^3 - \alpha^2 y^3 + z^3 - \alpha t^3 = 0.
$$

The group $G \cong \CG_3$ acts on $X$ as
$$
(x : y : z: t) \mapsto \left( x : y : \omega z : \omega^2 t \right).
$$

The triple of disjoint lines
$$
\begin{cases}
x = \sqrt[3]{\alpha}y, \\
z = \sqrt[3]{\alpha}t;
\end{cases} \qquad
\begin{cases}
x = \omega\sqrt[3]{\alpha}y, \\
z = \omega\sqrt[3]{\alpha}t;
\end{cases} \qquad
\begin{cases}
x = \omega^2\sqrt[3]{\alpha}y, \\
z = \omega^2\sqrt[3]{\alpha}t
\end{cases}
$$
\noindent is defined over $\ka$. Therefore $X$ is birationally equivalent to a del Pezzo surface of degree~$6$, and $\ka$-rational by Theorem \ref{ratcrit}, since the $\ka$-point $(1 : 0 : 0 : 1)$ lies on $X$.

Moreover, the $G$-fixed points $\left(\sqrt[3]{\alpha} : 1 : 0 : 0\right)$, $\left(\omega\sqrt[3]{\alpha} : 1 : 0 : 0\right)$ and $\left(\omega^2\sqrt[3]{\alpha} : 1 : 0 : 0\right)$ are transitively permuted by the Galois group. The $27$ lines on $X$ are given by the following equations
$$
\begin{cases}
x = \omega^i \sqrt[3]{\alpha}y, \\
z = \omega^j \sqrt[3]{\alpha}t;
\end{cases} \qquad
\begin{cases}
z = -\omega^i \sqrt[3]{\alpha}x, \\
t = -\omega^j \sqrt[3]{\alpha}y;
\end{cases} \qquad
\begin{cases}
x = \omega^i t, \\
z = \omega^j \sqrt[3]{\alpha^2}y,
\end{cases}
$$
\noindent where $i \in \{0, 1, 2\}$ and $j \in \{0, 1, 2\}$, and one can check that $\rho(X)^G = 1$. Therefore by Remark \ref{DP3C35min} the quotient $X / G$ is birationally equivalent to a minimal cubic surface $Y$ of type $\mathrm{VIII}$-$1$.

If the field $\ka$ does not contain $\omega$ then one can check that the same construction works for a cubic surface $X$ given by
$$
\alpha (P(x,y)) + P(z,t) = 0,
$$
\noindent where $\sqrt[3]{\alpha} \notin \ka$, and $P(z,t)$ is a cubic polynomial without roots defined over $\ka$ and invariant under the action of an element of order three, acting as
$$
(z, t) \mapsto (t, -z-t).
$$

In this case the quotient $X / G$ is birationally equivalent to a minimal cubic surface $Y$ of type $\mathrm{VIII}$-$2$.
\end{example}

To prove Theorem \ref{Galunirat} and Corollary \ref{Galuniratcor} we need the following definition.

\begin{definition}[cf. {\cite[Section 1]{Isk96}}]
\label{rigidness}
A minimal surface $X$ is called \textit{birationally rigid} if for any birational map of minimal surfaces $X \dashrightarrow X'$ one has $X' \cong X$.

\end{definition}

\begin{proof}[Proof of Theorem \ref{Galunirat}]

A Galois $\ka$-unirational surface $Y$ contains dense set of $\ka$-points, and is birationally equivalent either to a quotient of a surface admitting a structure of a conic bundle, or to a quotient of a $\ka$-rational $G$-minimal del Pezzo surface $X$ by a finite group~$G$. In the first case one can apply relative minimal model program and show that $Y$ is birationally equivalent to a surface admitting a structure of conic bundle. In the second case the surface $Y$ can be non-$\ka$-rational only if the surface $X$ and the group $G$ are listed in Theorem \ref{main}.

In this case $X$ is a del Pezzo of degree $d \leqslant 4$. If $G$ is trivial then $X$ is not $\ka$-rational by Theorem \ref{ratcrit}. If $G$ is not trivial then for all cases of Theorem \ref{main} except three the surface~$Y$ is birationally equivalent to a surface admitting a structure of a conic bundle by \cite[Lemmas 5.8, 6.1, 6.8]{Tr18a}, \cite[Lemmas 4.5, 4.16]{Tr18b} and Lemmas \ref{DP1typeII}, \ref{DP1typeIII} and~\ref{DP1typeIX}. These three cases are when $d = 3$ and $G$ is a certain group of order $3$, when $d = 2$ and $G$ is a certain group of order $2$, and when $d = 1$ and $G$ is a certain group of order $2$. For the cases $d = 2$ and $d = 1$ the surface $X$ is not $\ka$-rational by \cite[Lemma 5.1]{Tr18b} and \cite[Remark 5.2]{Tr18b}.

For the remaining case $d = 3$ the quotient is birationally equivalent to a cubic surface of type $\mathrm{VIII}$ by Lemma \ref{DP3C35}. Example \ref{DP3C35example} shows that this cubic surface can be minimal (and therefore birationally rigid) for any field $\ka$ with at least one non-cube.

\end{proof}

\begin{proof}[Proof of Corollary \ref{Galuniratcor}]

Let $X$ be a del Pezzo surface of degree $d$ such that $X(\ka) \neq \varnothing$ and $\rho(X) = 1$.

The surface $X$ is $\ka$-rational for $d \geqslant 5$ by Theorem \ref{ratcrit}, and thus $X$ is Galois $\ka$-unirational in this case. If $d = 4$ then $X$ is Galois $\ka$-unirational by \cite[Theorem IV.8.1]{Man74}.

If $d \leqslant 3$ then $X$ is birationally rigid by \cite[Theorems 1.6, 4.4, 4.5]{Isk96}. Therefore by Theorem \ref{Galunirat} in this case $X$ can be Galois $\ka$-rational only if $d = 3$ and $X$ is a cubic surface of type $\mathrm{VIII}$.

\end{proof}

Now we pass to the proof of Theorem \ref{Galgeomunirat}. 

\begin{proof}[Proof of Theorem \ref{Galgeomunirat}]

As in the proof of Theorem \ref{Galunirat} one can see that the quotient of a del Pezzo surface of by a nontrivial group is either birationally equivalent to a surface admitting a structure of a conic bundle (see Remark \ref{Cbundlebiratrem}), or birationally equivalent to a del Pezzo surface of degree no less than~$4$ (see Remark \ref{conditions}), or birationally equivalent to a del Pezzo surface of degree $3$ with an Eckardt point defined over $\ka$ (see Lemma \ref{DP3C35} and \cite[Lemma 3.2]{Tr16b}), or birationally equivalent to a del Pezzo surface of degree $2$ with a generalized Eckardt point defined over $\ka$ (see \cite[Lemma~4.8]{Tr18b}).

In particular, the other minimal del Pezzo surfaces of degree $3$, $2$ and $1$ are not birationally equivalent to a quotient of a geometrically rational surface, since these surfaces are birationally rigid.

\end{proof}

Now we prove Theorem \ref{Unquotient}.

\begin{proof}[Proof of Theorem \ref{Unquotient}]

Applying $G$-minimal model program we can assume that $X$ either admits a structure of a conic bundle, or is a del Pezzo surface.

If $X$ admits a structure of a conic bundle then the quotient $X / G$ is $\ka$-rational by Theorem \ref{Cbundlebirat}. If $X$ is a del Pezzo surface then the quotient $X / G$ can be non-$\ka$-rational only if $G$ is a cyclic group of order $3$ by Theorem \ref{main}.

\end{proof}

\bibliographystyle{alpha}

\end{document}